\documentclass[11pt, reqno]{amsart}
\usepackage{enumitem}
\usepackage{amsthm,amsmath,amssymb}
\usepackage{xcolor}
\usepackage{hyperref}
\hypersetup{colorlinks,citecolor=blue,linkcolor=blue,breaklinks=true}
\usepackage{mathrsfs}
\usepackage[textwidth=14.5cm]{geometry}
\usepackage{blindtext}
\usepackage{indentfirst}
\usepackage{braket}
\usepackage[utf8]{inputenc}
\usepackage[T1]{fontenc} 
\usepackage[utf8]{inputenc}
\usepackage{fancyhdr}
\numberwithin{equation}{section}

\newtheorem{theorem}{Theorem}[section]

\theoremstyle{plain}
\newtheorem{proposition}[theorem]{Proposition}
\newtheorem{lemma}[theorem]{Lemma}
\newtheorem{corollary}[theorem]{Corollary}
\newtheorem{remark}[theorem]{Remark}
\newtheorem{question}{Question}

\def\eps{\varepsilon}

\def\be{\begin{equation}}
	\def\ee{\end{equation}}

\def\eps{\varepsilon}

\def\pa{\partial}

\begin{document}

	\title[]{Heat kernel estimate on weighted Riemannian manifolds under lower $N$-Ricci curvature bounds with $\epsilon$-range and it's application}
	\author{Wen-Qi Li, Zhikai Zhang}
	\address{Wen-Qi Li and Zhikai Zhang, School of Mathematical Sciences, Key Laboratory of MEA (Ministry of Education) \& Shanghai Key Laboratory of PMMP, East China Normal University, Shanghai 200241, China}
	\email{51255500054@stu.ecnu.edu.cn, 51265500023@stu.ecnu.edu.cn}
	
	\date{\today}
	
	\begin{abstract}
		In this paper, we establish a parabolic Harnack inequality for positive solutions of the $\phi$-heat equation and prove Gaussian upper and lower bounds for the $\phi$-heat kernel on weighted Riemannian manifolds under  lower $N$-Ricci curvature bound with $\varepsilon$-range. Building on these results, we demonstrate:
		\begin{itemize}
			\item The $L^1_\phi$-Liouville theorem for $\phi$-subharmonic functions
			\item $L^1_\phi$-uniqueness property for solutions of the $\phi$-heat equation
			\item Lower bounds for eigenvalues of the weighted Laplacian $\Delta_\phi$
		\end{itemize}
		Furthermore, leveraging the Gaussian upper bound of the weighted heat kernel, we construct a Li-Yau-type gradient estimate for the positive solution of weighted heat equation under a weighted $L^p(\mu)$-norm constraint on $|\nabla\phi|^2$.
	\end{abstract}
	\maketitle
	\section{Introduction}
In recent years, the analysis of lower Ricci curvature bounds has become a cornerstone in geometric analysis, profoundly influencing our understanding of Riemannian manifolds. These curvature conditions not only govern fundamental geometric and topological properties \cite{Li}, but also play pivotal roles in diverse areas ranging from heat equation behavior to spectral theory. A particularly significant application lies in characterizing the volume growth of geodesic balls, leading to crucial geometric inequalities such as the Sobolev inequality.

When extending these considerations to weighted Riemannian manifolds (alternatively called smooth metric measure spaces), identifying appropriate geometric invariants becomes essential. The seminal work of Bakry and \'Emery \cite{BakryEmery1985} introduced the \textit{Bakry-\'Emery Ricci curvature tensor} $Ric_{\phi}^{m}$, where $m\geq n$ represents the effective dimension. This curvature quantity generalizes classical Ricci curvature while preserving key comparison geometry results (see \cite{BakryQian,Song2023,Wu2024,Wei2009,WW1,WW2,ZZ} for developments). Notably, manifolds with lower Bakry-\'Emery Ricci curvature bounds exhibit deep connections with:
\begin{itemize}
    \item Singularity formation in Ricci flow \cite{Hamilton1995,Pel1}
    \item Theory of metric measure spaces satisfies curvature dimension condition  \cite{LottVillani2009,Sturm1,Sturm2}
    \item Geometric analysis of stationary black holes \cite{Galloway2021,ZhangZhu2019}
\end{itemize}
This broad spectrum of applications underscores the fundamental nature of Bakry-\'Emery Ricci curvature in modern geometric analysis.

The validity of fundamental comparison theorems — including the Bishop-Gromov volume comparison and Laplacian comparison — under Bakry-Émery Ricci curvature bounds exhibits dimensional sensitivity: these results hold for effective dimensions $m \geq n$ but fail when $m \leq 1$. This limitation motivated three generations of curvature bound refinements:

\begin{enumerate}
    \item \textbf{Exponential Weighting} (Wylie-Yeroshkin \cite{YW}): For $m=1$, the variable curvature bound
    \begin{equation}\label{eq:WY-bound}
        Ric_{\phi}^{1} \geq Ke^{-\frac{4\phi}{n-1}}
    \end{equation}
    restores comparison theorems through exponential damping of the weight function $\phi$.
    
    \item \textbf{Negative Dimension Extension} (Kuwae-Li \cite{KL}): Generalized to $m \in (-\infty,0)$ via
    \begin{equation}\label{eq:KL-bound}
        Ric_{\phi}^{m} \geq Ke^{\frac{4\phi}{m-n}},
    \end{equation}
    leveraging reciprocal exponentiation for negative effective dimensions.
    
    \item \textbf{$\epsilon$-range case} (Lu-Minguzzi-Ohta \cite{LuMingOhta}): Introduced the $\varepsilon$-parametrized bound
    \begin{equation}\label{eq:LMO-bound}
        Ric_{\phi}^{m} \geq Ke^{\frac{4(\varepsilon-1)\phi}{m-n}},
    \end{equation}
   	which incorporates an additional parameter $\eps$ in an appropriate range, depending on $m\in(-\infty,1]\cup [n,+\infty] $, called the $\eps-$range. This framework subsumes both \eqref{eq:WY-bound} and \eqref{eq:KL-bound} through strategic $\varepsilon$-selection.
\end{enumerate}

 For further context, see \cite{LuMingOhta}, which discusses prior work on singularity theorems in Lorentz-Finsler geometry. This not only generalizes \cite{YW} and \cite{KL}, but also unifies both constant and variable curvature bounds by appropriately selecting $\eps$. For further investigations on the $\eps-$range, we refer to \cite{Kuwae2021,Kuwae2022,Kuwae2023}.

 In this paper, we investigate weighted heat kernel estimates on weighted Riemannian manifolds under curvature constraints characterized by lower $N$-Ricci bounds with $\varepsilon$-range, and explore their geometric applications. The study of heat kernel estimates under Ricci curvature conditions has a well-established history, with fundamental contributions including Gaussian bounds, Harnack inequalities, and mean value properties. We refer to the comprehensive treatments in \cite{Davies1989,G1,sal2,WW1,WW2} for historical context and classical results.

 Our first main result is the local $\phi-$heat kernel estimate Theorem \ref{hkestm}. The idea of the proof is standard. A key ingredient is a relative volume comparison theorem proved by Lu-Minguzzi-Otha \cite{Lu-Ohta}. Incorporating the local Sobolev inequality established by Y. Fujitani \cite{Fujitani2}, we can derive the parabolic mean value inequality \eqref{pbmvp} and the parabolic harnack inequality \eqref{pbhnk} by following the framework outlined in \cite{WW1,WW2}(for the context of Riemannian manifolds, one can refer to \cite{sal1, sal2}). The subsequent application of the mean value property, alongside Davies’ double integral estimate, will lead us to an upper bound for the $\phi-$heat kernel. Then with the Li-Yau type harnack inequality \eqref{LYhnk} and the upper bound, we can derive the lower bound estimate of the $\phi-$heat kernel.

\begin{theorem}[Weighted Heat Kernel Estimates]\label{hkestm}
Let $(M^n, g, \mu)$ be a weighted Riemannian manifold with a lower \(N\)-Ricci curvature bound and \(\varepsilon\)-range. Assume \(\phi\) satisfies 
		$$0 < a \le e^{\frac{2(1-\eps)\phi(x)}{n-1}} \le b.$$
		For all \(o \in M^n\), \(R > 0\), and \(\varepsilon > 0\), the \(\phi\)-heat kernel satisfies
		\begin{equation}
			\begin{aligned}
				\frac{C(\eps)E_{2}^{'} \exp\left(2D_2 \sqrt{K_{\eps}(q,10\sqrt{t})} \sqrt{t}\right)}{(V_{x}(\sqrt{t}))^{1/2} (V_{y}(\sqrt{t}))^{1/2}} &\exp\left(-\frac{d^2(x,y)}{4(1+\varepsilon)t}\right) \\
				& \ge H_{\phi}(x,y,t) \ge C^{'}_{12} \exp \left( -C^{'}_{13} t - \frac{C^{'}_{14}d^2(x,y)}{ t} \right) V_{x}^{-1}(\sqrt{t}),
			\end{aligned}
		\end{equation}
		where \(E_{2}^{'}, D_{2}\) are constants depending on \(a, b, n, \nu\), $C_{i}^{'}$ is a constant depends on $n,\ \nu,a,b,c,K$ and \(\lim_{\varepsilon \to 0} C(\varepsilon) = +\infty\).
\end{theorem}

	With the help of the heat kernel upper bound, one can prove the $L^{1}_{\phi}-$Liouville property for non-negative subharmonic functions by following the argument in \cite{Li1984,LiSchoen1984}.
	While \cite{Fujitani2} derived several Liouville theorems by examining the lower bounds of volume growth, those results required specific assumptions on the function $\phi$.  Here, by leveraging the $\phi-$heat kernel estimate, we eliminate these assumptions and successfully derive the $L^{1}_{\phi}-$Liouville Theorem.
	\begin{theorem}\label{liouville}
		Let $(M^n, g, \mu)$ be a complete noncompact weighted Riemannian manifold under lower \(N\)-Ricci curvature bound with \(\varepsilon\)-range. Assume \(\phi\) satisfies 
		$$0 < a \le e^{\frac{2(1-\eps)\phi(x)}{n-1}} \le b.$$
		Then any non-negative $L^{1}(\mu)$ integrable $\phi-$subharmonic function on $M^n$ must be identically constant.  In
		particular, any $L^1(\mu)-$integrable harmonic function must be identically constant.
	\end{theorem}
 By the $L^{1}_\phi-$Liouville Theorem, following the arguments of Li in \cite{Li1984}, we can prove a uniqueness Theorem \ref{L1uniq} for $L^{1}_{\phi}-$solution of the $\phi-$heat equation. Another application of the weighted heat kernel estimate is that one can get Li-Yau type \cite{LiYau1986} lower bound estimates for the eigenvalues of $\Delta_\phi$ as Theorem \ref{lbeg} and Theorem \ref{lbeg2}.
	
Finally, building upon the methodology developed in \cite{ZhangZhu2018}, we establish a Li-Yau-type gradient estimate for positive solutions of the weighted heat equation under a weighted $L^p(\mu)$-norm constraint on $|\nabla\phi|^2$ as follows.
\begin{theorem}\label{LYGEWHE}
	Let $(M^n, g, \mu)$ be a compact weighted Riemannian manifold satisfying a lower $N$-Ricci curvature bound in the $\varepsilon$-range. Assume the potential function $\phi$ satisfies:
	\begin{equation}
		0 < a \leqslant e^{\frac{2(1-\varepsilon)\phi(x)}{n-1}} \leqslant b
	\end{equation}
	with the gradient bound $\|\nabla\phi\|_{L^p(\mu)} \leq V$ for some $V > 0$ , $\alpha >1$ and $p> n$. Then for any positive solution $u$ of the weighted heat equation, the following Li-Yau-type estimate holds:
	\begin{equation}
		 \underline{J}(t)\frac{|\nabla u|^2}{u^2} -\alpha \frac{\partial_t u}{u} \leq \frac{2n+1}{2t\underline{J}(t)},
	\end{equation}
where
\begin{equation}
	\underline{J}(t) = 2^{-\frac{1}{\tau-1}} e^{-(\tau-1)^{\frac{n}{(2p-n)}} (4 C_{29} \hat{C}(t)^{\frac{1}{p}})^{\frac{2p}{2p-n}} t}
\end{equation} 
and
$$
\begin{aligned}
	C_{29} &= K \| e^{\frac{4(\varepsilon-1) \varphi(x)}{n-1}} \|_p + C_{28} \| |\nabla \varphi|^2 \|_p \\
	\hat{C}(t) &= C(\varepsilon) E_2^{'} \exp \left(2 D_2 \sqrt{K_\varepsilon(q,10\sqrt{t})} \sqrt{t}\right)
\end{aligned}
$$
\end{theorem}
 This result provides a partial resolution to the open problem concerning negative-dimensional $N-$Ricci curvature regimes posed by S. Ohta in \cite{Ohta2016,Ohta2021}.

The paper is structured as follows: 
\begin{itemize}
	\item Section~2 reviews fundamental concepts including $N$-Ricci curvature bounds, $\varepsilon$-range geometric conditions, and presents the volume comparison theorem from \cite{Lu-Ohta} coupled with Fujitani's local Sobolev inequality \cite{Fujitani2}
	\item Section~3 develops the parabolic mean value inequality through adaptation of techniques from \cite{sal2,WW2}
	\item Section~4 establishes the Gaussian upper and lower bounds for the $\phi$-heat kernel
	\item Section~5 contains the proofs of:
	\begin{itemize}
		\item The $L^1_\phi$-Liouville theorem for $\phi$-subharmonic functions
		\item $L^1_\phi$-uniqueness property for the $\phi$-heat equation
	\end{itemize}
	\item Section~6 derives spectral estimates for $\Delta_\phi$, providing the eigenvalue bounds
	\item Section~7 constructs the complete Li-Yau-type estimate through heat kernel estimate
\end{itemize}
	\section{preliminaries}
	In this paper, we are primarily concerned with the weighted Riemannian manifolds $(M^n, g, \mu)$ where $(M^n,g)$ is a complete Riemannian manifolds and $\nu$ is a Radon measure which can be represented as $d\mu=e^{-\phi}d\nu_{g}$ by the Riemannian volume measure $d\nu_{g}$ and a smooth function $\phi$. The $N-Ricci\ curvature$ of $(M^n, g, \mu)$ is defined as follows:
	$$ Ric_{\phi}^{N}:=Ric+\nabla^{2}\phi-\frac{d\phi\otimes d\phi}{N-n},$$
	for $N\in (-\infty, 1]\bigcup [n,+\infty]$, where when $N=n$, we only consider $\phi\equiv C$, which means $Ric_{\phi}^{N}=Ric$ and when $N=+\infty$ the last term becomes 0. The $N-Ricci \ curvature$ is an extension of Bakry-\'Emery Ricci curvature and it is easy to check that when $N\in [n,+\infty]$, it is equal to the Bakry-\'Emery Ricci curvature.
	
	In \cite{Lu-Ohta} , Lu-Minguzzi-Otha introduced a new notation called $\epsilon-range$, which is defined as
	\begin{equation}\label{epsrange}
		|\eps|<\sqrt{\frac{N-1}{N-n}}\text{, } \eps=0\text{, for}\ N=1 \text{ and } \eps\in\mathbb{R}\text{ for
		} N=n.
	\end{equation}
	With this, they give a new lower $N-Ricci\ curvature$ bound condition as 
	\begin{equation}
		Ric_{\phi}^{N}(v)\ge Ke^{\frac{4(\eps-1)\phi(x)}{n-1}}g(v,v)
	\end{equation}
	for $K\in \mathbb{R}$ and $v\in T_x(M)$. In addition, we define the constant $c$ associated with $\eps$ as
	$$c=\frac{1}{n-1}(1-\eps^2\frac{N-n}{N-1})>0$$
	for $N\neq 1$ and $c=(n-1)^{-1}$ when $N=1$. We also set 
	$$K_{\eps}(x)=max\{0,\ \underset{v\in U_xM^n}{sup}(-e^{\frac{-4(\eps-1)}{n-1}\phi(x)})Ric_{\phi}^{N}(v)\}$$
	and
	$$K_{\eps}(q,R)=\underset{x\in B_q(R)}{sup}K_{\eps}(x)$$
	for $q\in M^n$ and $R>0$, where $U_xM^n$ denotes the unit vector in $T_xM^n$.
	
	The reason they introduce this condition is to derive a new weighted volume comparison theorem. To provide the representation of this theorem, we need to first introduce several new concepts. Where the first concept is the comparison function $\mathbf{s}_{\kappa}$, which is defined as
	\begin{equation}
		\mathbf{s}_{\sqrt{K}}(t)=
		\begin{cases}
			\frac{\sin(\sqrt{K}t )}{\sqrt{K}} & \text{if } K > 0, \\
			t & \text{if } K = 0, \\
			\frac{\sinh(\sqrt{-K} t)}{\sqrt{-K}} & \text{if } K < 0.
		\end{cases}
	\end{equation} 
	In this paper, the notation $V_x(R)$ is used to describe the measure of the geodesic ball $B_x(R)$, which means $V_x(R)=\mu(B_x(R))$ and $V_x(r, R)=\mu(B_x(R)\setminus B_x(r))$ for $0<r<R$.  
	The $weighted\ Laplacian\ \Delta_\phi$ on $(M^n, g, \mu)$ is defined as 
	\begin{equation}
		\Delta_\phi=\Delta-\langle\nabla\phi, \nabla\rangle.
	\end{equation} 
	Therefore, a function $u$ called $\phi-harmonic$ if $\Delta_\phi u=0$ and $\phi-subharmonic$ if $\Delta_\phi u\ge0$. It is easy to check that $\Delta_\phi$ is self-adjoint with respect to the measure $\mu$. 
	
	In \cite{Lu-Ohta}, Lu-Minguzzi-Otha prove that the $\phi$-Laplacian comparison under the condition that the weighted Riemannian manifold $(M^n, g, \mu)$ possesses the lower $N-Ricci\ $curvature bound with $\eps-range$.
	\begin{theorem}\label{lpcp}
		Let $(M^n, g, \mu)$ be a weighted Riemannian manifold be forward complete and under lower $N-Ricci\ curvature$ bound with $\eps-range$, which means 
		$$Ric_{\phi}^{N}(v)\ge Ke^{\frac{4(\eps-1)\phi(x)}{n-1}}g(v,v)$$
		for all $\eps$ satisfies condition \eqref{epsrange}, $v\in T_x(M^n)\textbackslash 0$ and a smooth function $\phi$. 
		
		If $\phi$ satisfies 
		$$0<a\le e^{\frac{2(1-\eps)\phi(x)}{n-1}}\le b,$$
		then for $\forall p\in M^n$, the distance function $r(x)=d(x, p)$ satisfies
		$$\Delta_\phi r(x)\le \frac{1}{c\rho}\frac{\mathbf{s}_{cK}(r(x)/b)'}{\mathbf{s}_{cK}(r(x)/b)},$$
		on $M^n\textbackslash(\{p\}\bigcup Cut(p))$. Here, $\rho$ takes the value $a$ when $\mathbf{s}_{c\kappa}(t)'\ge0$ and $b$
		when $\mathbf{s}_{c\kappa}(t)'<0$. The notation $Cut(p)$ denotes the cut locus of $p$. 
	\end{theorem}
	\begin{remark}
		In all subsequent part of this paper, we refer to $(M^n, g, \mu)$ a weighted Riemannian manifold under lower $N-Ricci\ curvature$ bound with $\eps-range$ if it satisfies the same conditions outlined in Theorem \ref{lpcp}.
	\end{remark}
	
	With this, they derive the Bishop-Gromov type volume comparison theorem
	\begin{theorem}\label{VCP}
		Let $(M^n, g, \mu)$ be a weighted Riemannian manifold satisfies same conditions as Theorem \ref{lpcp}. Then, for all $x\in M^n$ and $0<r\le R$, we have the inequality
		$$\frac{V_x(R)}{V_x(r)}\le \frac{b\int_{0}^{min\{R/a,\ \pi/\sqrt{cK}\}}\mathbf{s}_{cK}(t)^{1/c}dt}{a\int_{0}^{r/b}\mathbf{s}_{cK}(t)^{1/c}dt}$$
		. Here R must satisfy $R\le \frac{b\pi}{c\sqrt{K}}$ if $K>0$, while $\frac{b\pi}{c\sqrt{K}}=+\infty$ when $K\le 0$.
	\end{theorem}
	Using the Bishop-Gromov type volume comparison theorem, we can derive two essential tools. The first is an estimate for the volume ratio of geodesic balls with the same center. 
	\begin{corollary}\label{col1}
		Under the same assumptions as Theorem \ref{VCP}, we have 
		$$ \frac{V_x(r)}{V_x(s)}\le (\frac{b}{a})^{(1+2c)/c}(\frac{r}{s})^{(1+c)/c}exp(\sqrt{\frac{K_1}{c}}\frac{r}{a})$$
		for all $x\in M^n$, where $K_{1}=max\{0, -K\}$.
	\end{corollary} 
	\begin{proof}
		We know that 
		\begin{equation}
			\frac{\mathbf{s}_{-cK_1}(r)}{\mathbf{s}_{-cK_1}(s)}=\exp(\int_{s}^{r}\sqrt{cK_1}coth(\sqrt{cK_1}t)dt).
		\end{equation}
		Hence, by direct computations, we obtain
		\begin{equation}
			\exp(\int_{s}^{r}\sqrt{cK_1}coth(\sqrt{cK_1}t)dt)\le \exp(\int_{s}^{r}(\frac{1}{t}+\sqrt{cK_1})dt)\le\frac{r}{s}\exp(\sqrt{cK_1}r),
		\end{equation}
		Which means 
		\begin{equation}
			\mathbf{s}_{-cK_1}(r)\le\mathbf{s}_{-cK_1}(s)\frac{r}{s}\exp(\sqrt{cK_1}r).
		\end{equation}
		Then by volume comparison theorem, it is easy to prove that 
		\begin{equation}
			\begin{aligned}
				\frac{V_x(r)}{V_x(s)}&\le\frac{b\int_{0}^{min\{r/a,\ \pi/\sqrt{cK}\}}\mathbf{s}_{cK}(t)^{1/c}dt}{a\int_{0}^{s/b}\mathbf{s}_{cK}(t)^{1/c}dt}
				\\&\le\frac{b\int_{0}^{min\{r/a,\ \pi/\sqrt{-cK_1}\}}\mathbf{s}_{-cK_1}(t)^{1/c}dt}{a\int_{0}^{s/b}\mathbf{s}_{-cK_1}(t)^{1/c}dt}
				\\&\le(\frac{b}{a})^{(1+2c)/c}(\frac{r}{s})^{(1+c)/c}\exp(\sqrt{\frac{K_1}{c}}\frac{r}{a}).
			\end{aligned}
		\end{equation}
	\end{proof}
	Using Corollary \ref{col1}, we can derive the volume doubling property, which states:
	\begin{equation}\label{vldb}
		\frac{V_x(2R_1)}{V_x(R_1)}\le (\frac{b}{a})^{(1+2c)/c}(2)^{(1+c)/c}exp(\sqrt{\frac{K_1}{c}}\frac{2R_1}{a})
	\end{equation}
	for all $x\in M^n$ and $2R_1\le \frac{b\pi}{cK}$. 
	
	Another is the volume ratio of geodesic ball with same radius and different center.
	\begin{corollary}\label{col2}
		Let $x,\ y\in M^n$ be two distinct points and let $d=d(x, y)$ denote the distance between them. Then we have 
		$$ \frac{V_x(s)}{V_y(s)}\le (\frac{b}{a})^{(1+2c)/c}(\frac{s+d}{s})^{(1+c)/c}exp(\sqrt{\frac{K_1}{c}}\frac{s+d}{a}).$$
	\end{corollary}
	This result follows easily by substituting $r=s+d$ into the inequality from Corollary \ref{col1}.
	
	In \cite{Fujitani2}, the author utilized the volume comparison theorem alongside Corollaries \ref{col1} and \ref{col2} to establish the Neumann-Poincaré inequality. 
	\begin{theorem}\label{npineq}
		Let $(M^n, g, \mu)$ be a weighted Riemannian manifold under lower $N-Ricci\ curvature$ bound with $\eps-range$. If $\phi$ satisfies 
		$$0<a\le e^{\frac{2(\eps-1)\phi(x)}{N-1}}\le b,$$ then for $\varphi\in C_0(M^n)$, we have
		\begin{equation}
			\int_{B_o(R)}|\varphi-\varphi_{B_o(R)}|^2d\mu\le 2^{(n+3)}(\frac{2b}{a})^{\frac{1}{c}}exp(\sqrt{\frac{K_{\eps}(o,2R)}{c}}\frac{2R}{a})R^2\int_{B_o(2R)}|\nabla\varphi|^2d\mu,
		\end{equation}
		where $\varphi_{B_o(R)}=\frac{\int_{B_o(R)}\varphi d\mu}{\mu{(B_o(R))}}$.
	\end{theorem}
	\begin{remark}
		For $1<\theta\le 2$, we can indeed establish a similar inequality: 
		$$	\int_{B_o(R)}|\varphi-\varphi_{B_o(R)}|^2d\mu\le P_{\theta}exp(\sqrt{\frac{K_{\eps}(o,\theta R)}{c}}\frac{\theta R}{a})R^2\int_{B_o(\theta R)}|\nabla\varphi|^2d\mu,$$
		where $P_{\theta}$ is a constant that depends on $a,\ b,\ c,\ n,\ \theta$.
	\end{remark}
	
	At last, with volume doubling property and Neumann-poincar\'e inequality, Fujitani \cite{Fujitani2} prove that following local sobolev inequality holds.
	\begin{theorem}
		Under the same assumption as Theorem \ref{npineq}, there exists positive constants $D_1$, $E_1$ depending on $a,\ b,\ c,\ n$ such that
		\begin{equation}\label{lcsblv}
			\small{	\left(\int_{B_q(R)} |\varphi|^{\frac{2\nu}{\nu-2}} \, d\mu \right)^{\frac{\nu-2}{\nu}} \le E_1 \exp\left(D_1 \sqrt{K_{\eps}(q,10R)} R\right) R^2 V_q(R)^{-\frac{2}{\nu}} \int_{B_q(R)} \left( |\nabla \varphi|^2 + R^{-2} \varphi^2 \right) d\mu}
		\end{equation}
		for all $B_q(R)\in M^n$ and $\varphi\in C_{0}^{\infty}(B_q(R))$, where
		\begin{equation}
			\nu=
			\begin{cases}
				3& \text{if } c=1, \\
				1+\frac{1}{c} & \text{if } c < 1.
			\end{cases}
		\end{equation} 
	\end{theorem} 
	\begin{remark}
		Moreover, the coefficients are defined as $E_1=c_{1}(n,\ \nu)(\frac{b}{a})^{c_{2}(\nu)}$ and $D_{1}=\sqrt{\frac{c_{3}(\nu)}{a^{2}c}}$, where in this remark and the following part of this paper, the notation $c_{i}(\nu)$ means constant depends on $\nu$. So when $\phi\equiv C$, considers $a= e^{\frac{2(\eps-1)C}{N-1}}=b$ and $N=n$, then the coefficients only depend on $n$ and $\eps$, which means this inequality can be reduced to that in the lower Ricci curvature bound condition.
	\end{remark}
	
	\section{Moser's Harnack inequality for $\phi-$Heat equation}
	In this section, we will prove Moser's Harnack inequality for $\phi-$heat equation using moser iteration, which will be used to derivce the lower bound estimate of of $\phi-$heat kernal.
	
	First, we introduce the $\phi-$heat equation and $\phi-$Heat kernal. Analogous to the classical heat equation, $\phi-$heat equation is defined as 
	$$(\pa_t-\Delta_\phi)u=0.$$
	We denote the $\phi-$heat kernal as $H_{\phi}(x,y,t)$, where for each $y\in M^n$ $u(x,t)=H_{\phi}(x,y,t)$ represents the minimal positive solution of $\phi-$heat equation with the intital condition 
	$$\lim_{t\to0}u(x,t)=\delta_{\phi,y}(x),$$ where $\delta_{\phi,y}(x)$ denotes the Dirac measure with respect to the measure $\mu$ in this paper. In this paper, we focus on the weighted $L_p$ space $L_p(M^n,\mu)$ with norm $||u||_p=(\int_{M^n}|u|^pd\mu)^{\frac{1}{p}}$.
	
	In the previous section, we established the local Sobolev inequality \eqref{lcsblv}. Building on this result, along with the findings in \cite{sal2,WW2}, we obtain the following mean value inequality for the solution of the $\phi-$heat equation.
	\begin{proposition}
		Let $(M^n, g, \mu)$ be a weighted Riemannian manifold under lower $N-Ricci\ curvature$ bound with $\eps-range$. Assume $\phi$ satisfies 
		$$0<a\le e^{\frac{2(1-\eps)\phi(x)}{n-1}}\le b.$$ 
		Fix $0<p<\infty$. Then there exists positive constants $D_2=\sqrt{{\frac{c_{4}(\nu,p,n)}{a^{2}c}}}$, $E_2=c_{5}(\nu,p,n)(\frac{b}{a})^{c_{6}(\nu,p,n)}$  such that for any $s\in \mathbb{R}$ and $0<\delta<1$, any smooth positive solution $u$ of $\phi-$heat equation in cylinder $Q=B_x(R)\times (s-R^2,s)$ satisfies
		\begin{equation}\label{pbmvp}
			\underset{Q_{\delta}}{sup}\{u^p\}\le\frac{E_2 \exp\left(D_2 \sqrt{K_{\eps}(q,10R)} R\right)}{(1-\delta)^{2+\nu}R^{2}V_x(R)}\int_{Q}u^pd\mu dt,
		\end{equation}
		where $Q_\delta=B_x(\delta R)\times (s-\delta R^2,s)$ and 
		\begin{equation}
			\nu=
			\begin{cases}
				3& \text{if } c=1, \\
				1+\frac{1}{c} & \text{if } c < 1.
			\end{cases}
		\end{equation} .
	\end{proposition}
	Similarly, we can also prove that
	\begin{proposition}
		Under the same assumption as Proposition \ref{pbmvp}. Fix $0<p_0<1+\nu/2$. Then, there exist positive constants $D_3=\sqrt{{\frac{c_{7}(\nu,p_0,n)}{a^{2}c}}}$ and $E_3=c_{8}(\nu,p_0,n)(\frac{b}{a})^{c_{9}(\nu,p_0,n)}$ such that for any $s\in \mathbb{R}$ and $0<\delta<1$, and $0<p\le p_0$, any smooth positive solution $u$ of $\phi-$heat equation in cylinder $Q=B_x(R)\times (s-R^2,s)$ satisfies
		\begin{equation}
			||u||^{p}_{p_0, Q^{'}_{\delta}}\le\{\frac{E_3 \exp\left(D_3 \sqrt{K_{\eps}(q,10R)} R\right)}{(1-\delta)^{2+\nu}R^{2}V_x(R)}\}^{1-p/p_0}||u||^{p}_{p, Q},
		\end{equation}
		where $Q'_\delta=B_x(\delta R)\times (s-R^2,s-(1-\delta)R^2)$ and 
		\begin{equation}
			\nu=
			\begin{cases}
				3& \text{if } c=1, \\
				1+\frac{1}{c} & \text{if } c < 1.
			\end{cases}
		\end{equation} 
		
		On the other hand, for any $0<p<\bar{p}<\infty$ There exists positive constants $D_4=\sqrt{{\frac{c_{10}(\nu,\bar{p},n)}{a^{2}c}}}$, $E_4=c_{11}(\nu,\bar{p},n)(\frac{b}{a})^{c_{12}(\nu,\bar{p},n)}$  such that 
		\begin{equation}
			\underset{Q_{\delta}}{sup}\{u^{-p}\}\le\frac{E_4 \exp\left(D_4 \sqrt{K_{\eps}(q,10R)} R\right)}{(1-\delta)^{2+\nu}R^{2}V_x(R)}\int_{Q}u^{-p}d\mu dt.
		\end{equation}
	\end{proposition}
	\begin{remark}
		The above two propositions also apply to upper solutions of the $\phi-$ heat equation.
	\end{remark}
	With previous two propositions, we will now prove the main result of this section. First, we need to establish some weighted Poincar\'e inequality. In last section, we prove the volume doubling property \eqref{vldb} and the Neumann-poincar\'e inequality \eqref{npineq}. By applying Theorem 5.3.4 from \cite{sal2}, we can assert that for a non-increasing function $\lambda:[0,\infty)\to [0,1]$ satisfies $\lambda(t)=0$ for $t>1$ and 
	\begin{equation}
		\lambda(s)=
		\begin{cases}
			1& \text{if } s\in[0,\delta], \\
			\frac{1-s}{1-\delta} & \text{if } s\in[\delta,1].
		\end{cases}
	\end{equation}
	following weighted Poincar\'e inequality holds. 
	
	\begin{lemma}
		Let $(M^n, g, \mu)$ be a weighted Riemannian manifold under lower $N-Ricci\ curvature$ bound with $\eps-range$. Assume $\phi$ satisfies 
		$$0<a\le e^{\frac{2(1-\eps)\phi(x)}{n-1}}\le b.$$ Then there exist positive constants $D_5=\sqrt{{\frac{c_{21}(\nu,n)}{a^{2}c}}}$ and $E_5=c_{22}(\nu,n,\delta)(\frac{b}{a})^{c_{23}(\nu,n)}$ such that for $\Psi(x):=\lambda(d(x,x_0)/R)$, the follwing weighted Poincar\'e inequality holds
		\begin{equation}\label{wpcr}
			\int_{E}|f-f_{\Psi}|^2\Psi d\mu\le E_{5}exp(D_5\sqrt{K_{\eps}}R)R^2\int_{E}|\nabla f|^2\Psi d\mu,
		\end{equation}
		where $f_{\Psi}=\int_{B_{x_o}(R)}f\Psi d\mu/\int_{B_{x_o}(R)}\Psi d\mu$.
	\end{lemma}
	\begin{proof}[(sketch of the proof):]
		First, the volume doubling property indicates that for any fixed ball $E=B(x_0, R)$, the Whitney type coverring $F(B)$ of $E$ exists, as shown in \cite{sal2}. Next, we have the following inequality:
		\begin{equation}\label{T1T2}
			\int_{E}|f-f_{4B_0}|^2\Psi d\mu\le 4\sum_{B\in F}\int_{4B}|f-f_{4B}|^2\Psi d\mu+4\sum_{B\in F}\int_{4B}|f_{4B}-f_{4B_0}|^2\Psi d\mu\equiv T_1+T_2,
		\end{equation}
		where $B_0\in F(B)$ is a ball containing $x_0$. 
		Denotes $\Psi(S)\equiv \int_{S}\Psi d\mu$ for any measurable set in $E$, there exists $C_4$ depends on $n\ \text{and}\ \beta$ such that 
		$$\frac{\Psi(B)}{V(B)}\le C_4\Psi(x),\ \ \ x\in E.$$
		It is straightforward to verify that 
		$$T_1\le C_1exp(C_2\sqrt{K_{\eps}}R)R^2\int_{E}|\nabla f|^2\Psi d\mu,$$
		where $C_1=c_{12}(\nu, n, \delta)(\frac{b}{a})^{c_{13}(\nu,n)}$, $C_2=\sqrt{{\frac{c_{14}(\nu,n)}{a^{2}c}}}$. 
		
		In addition, consider two adjacent balls $B_i, B_{i+1}\in F(B)$. Using inequality \eqref{npineq}, we have
		\begin{equation}\label{ineqcb}
			|f_{4B_i}-f_{4B_{i+1}}|^2\le C_3exp(\sqrt{\frac{K_{\eps}(q,2R)}{c}}\frac{2R}{a})\frac{r^2(B_i)}{V(B_i)}\int_{32B_{i}}|\nabla f|^2d\mu,
		\end{equation}
		where $4B_i$ denotes the ball with the same center as $B_i$ but with four times the radius. Here, $C_3=c_{14}(\nu, n)(\frac{b}{a})^{c_{15}(\nu,n)}$. 
		
		Following the same way as section $5.3.4$ in \cite{sal2}, we can also derive the maximal inequality as 
		\begin{equation}\label{maxineq}
			||Mf||_p\le C_5exp(C_2\sqrt{K_{\eps}}R)||f||_p,
		\end{equation}
		where $C_5=c_{16}(\nu, n)(\frac{b}{a})^{c_{17}(\nu,n)}$ and $Mf(x)=\underset{0<r\le R}{sup} \frac{\int_{B_x(r)}fd\mu}{V_x(r)}$.
		
		With inequalities \eqref{ineqcb} and \eqref{maxineq}, following the same argument in section $5.3.5$ in \cite{sal2}, we can derive the estimate of $T_2$ as 
		\begin{equation}
			T_2\le C_6exp(C_7\sqrt{K_{\eps}}R)R^2\int_{E}|\nabla f|^2\Psi d\mu,
		\end{equation}
		where $C_6=c_{18}(\nu, n,\delta)(\frac{b}{a})^{c_{19}(\nu,n)}$, $C_7=\sqrt{{\frac{c_{20}(\nu,n)}{a^{2}c}}}$.
		
		Finally, we get the weighted Poincar\'e inequality as 
		\begin{equation}
			\int_{E}|f-f_{\Psi}|^2\Psi d\mu\le\int_{E}|f-f_{4B_0}|^2\Psi d\mu\le E_{5}exp(D_5\sqrt{K_{\eps}}R)R^2\int_{E}|\nabla f|^2\Psi d\mu.
		\end{equation}
	\end{proof}
	Secondly, with the weighted poincar\'e inequality \eqref{wpcr}, we will derive the estimate of the weak $L^1(\mu)$ bound of $log u$, where $u$ is a positive solution of $\phi-$heat equation. In the following, the notation $d\bar{\mu}=d\mu dt$ denotes the measure of space-time.
	\begin{lemma}\label{wkLp}
		Let $(M^n, g, \mu)$ be a weighted Riemannian manifold under lower $N-Ricci\ curvature$ bound with $\eps-range$. Assume $\phi$ satisfies 
		$$0<a\le e^{\frac{2(1-\eps)\phi(x)}{n-1}}\le b.$$ Fix $s\in\mathbb{R}$ and $\delta,\ \rho\in (0,1)$. Denote 
		$$R_{+}=B_{x_0}(\delta R)\times[t_0-\rho R^2,\ t_0],\ R_{-}=B_{x_0}(\delta R)\times[t_0-R^2,\ t_0-\rho R^2].$$ Then for any smooth positive solution $u$ of $\phi-$heat equation and $\lambda>0$, we have 
		$$\bar{\mu}(\{(x,\ t)\in R_{-}|\ log\ u>\lambda-k\})\le V_1\lambda^{-1},$$
		$$\bar{\mu}(\{(x,\ t)\in R_{+}|\ log\ u^{-1}>k+\lambda\})\le V_1\lambda^{-1},$$ 
		where $V_{1}=2E_{5}exp(D_5\sqrt{K_{\eps}}R)R^2V(B_{x_0}(R))$.
	\end{lemma}
	\begin{proof}
		In the above lemma, we are discussing a solution of a fixed ball $B_{x_0}(R)$, so without loss of generality, we can assume $u$ is a positive solution. Let $\eta:= -\log u$, it is easy to check that 
		$$\Delta_\phi \eta-\pa_t\eta-|\nabla \eta|^2= 0.$$ 
		Then for any nonnegative function $\omega\in C_0(B_{x_0}(R))$, we have
		
		\begin{equation}
			\begin{aligned}
				\pa_t\int_{B_{x_o}(R)}\omega^2\eta d\mu&=\int_{B_{x_o}(R)}\omega^2(\Delta_\phi\eta-|\nabla \eta|^2) d\mu\\
				&=\int_{B_{x_o}(R)}-2\omega\Braket{\nabla\eta, \nabla\omega}-\omega^2|\nabla \eta|^2 d\mu\\
				&\le 2\int_{B_{x_o}(R)}|\nabla \omega|^2 d\mu-1/2\int_{B_{x_o}(R)}|\nabla \eta|^2\omega^2d\mu.
			\end{aligned}
		\end{equation}
		Let $\omega=\lambda(d(x_0,x)/R)$, where $\lambda:[0,1]\to[0,1]$ defined as
		\begin{equation}
			\lambda(s)=
			\begin{cases}
				1& \text{if } s\in[0,\delta], \\
				\frac{1-s}{1-\delta} & \text{if } s\in[\delta,1].
			\end{cases}
		\end{equation} 
		Then with weighted poincar\'e inequality \eqref{wpcr}, we can derive that
		\begin{equation}\label{ineqA}
			\pa_t\int_{B_{x_o}(R)}\omega^2\eta d\mu+(2E_{5}(D_5\sqrt{K_{\eps}}R)R^2)^{-1}\int_{E}|\eta-\eta_{\omega^2}|^2\omega^2 d\mu\le\frac{2V(B_{x_0}(R))}{[(1-\delta)R]^2}.
		\end{equation}
		For fixed $\delta$, $\int_{B_{x_o}(R)}\omega^2 d\mu=c_{21}(\delta)V(B_{x_0}(R))$, then inequality \eqref{ineqA} can be transform to
		\begin{equation}
			\pa_t\eta_{\omega^2}+(2E_{5}(D_5\sqrt{K_{\eps}}R)R^2V(B_{x_0}(R)))^{-1}\int_{E}|\eta-\eta_{\omega^2}|^2\omega^2 d\mu\le\frac{2}{[(1-\delta)R]^2}.
		\end{equation}
		Fixed $t_1=t_0-\rho R^2$ for some constant $\rho\in (0,1)$ and denote $\bar {\eta}=\eta-V_2(t-t_1)$, where $V_{1}=2E_{5}exp(D_5\sqrt{K_{\eps}}R)R^2V(B_{x_0}(R))$ and $V_2=\frac{2}{[(1-\delta)R]^2}$. Then we can obtain that 
		\begin{equation}
			\pa_t\bar{\eta}_{\omega^2}+(V_1)^{-1}\int_{E}|\eta-\eta_{\omega^2}|^2\omega^2 d\mu\le 0.
		\end{equation}
		Let $k=\bar{\eta}_{\omega^2}(t_1)$, for fixed $\lambda>0$ and $t\in [t_0-r^2, t_0]$, denote 
		$D^{+}_{t}(\lambda)=\{x\in B_{x_0}(R)|\ \bar{\eta}(x,t)>k+\lambda\}$ and $D^{-}_{t}(\lambda)=\{x\in B_{x_0}(R)|\ \bar{\eta}(x,t)<k-\lambda\}$.
		
		If $t>t_1$, then 
		$$\bar{\eta}(x,t)-\bar{\eta}_{\omega^2}(t)>k+\lambda-\bar{\eta}_{\omega^2}(t)\ge\lambda$$
		in $D^{+}_{t}(\lambda)$, since $\pa_t\bar{\eta}_{\omega^2}(t)\le 0$. Denote $g(t)=\bar{\eta}_{\omega^2}-(\lambda+a)$, we have
		\begin{equation}\label{ineq2}
			\dot{g}(t)+V_{1}^{-1}\mu(D^{+}_{t}(\lambda))g^2(t)\le 0.
		\end{equation}
		We know that $g(t_1)=-\lambda<0$, which means $g(t)<0$ when $t\ge t_1$. Then integral inequality  \eqref{ineq2} we have
		$$\int_{t_1}^{t_0}\mu(D^{+}_{t}(\lambda))dt\le V_{1}\lambda^{-1}.$$
		We know that $\bar {\eta}=\eta-V_2(t-t_1)$, then we can prove that 
		$$\bar{\mu}(\{(x,\ t)\in R_{+}|\ log\ u^{-1}>a+\lambda\})\le V_1\lambda^{-1}$$ 
		because $0<V_2(t-t_1)\le \frac{\rho}{1-\delta^2}$.
		
		Similarly, if $t<t_1$, then we have 
		$$\bar{\eta}(x,t)-\bar{\eta}_{\omega^2}(t)<a-\lambda-\bar{\eta}_{\omega^2}(t)\le-\lambda$$
		in $D^{-}_{t}(\lambda)$. Follow the same argument we can prove that 
		$$\bar{\mu}(\{(x,\ t)\in R_{-}|\ log\ u>\lambda-a\})\le V_1\lambda^{-1}.$$ 
	\end{proof}
	Thirdly, we need an interpolation inequality of $L^p(\mu)$ norm, which is mentioned in \cite{WW2}.
	\begin{lemma}\label{interpolation}
		Let $\gamma$, $C$, $1/2\le\delta<1$, $p_1<p_0\le\infty$ be positive constants, $R_{\sigma}$ is a measurable set in $M\times\mathbb{R}$ satisfies $R_{\sigma'}\subset R_{\sigma}$ if $\sigma\ge\sigma'$. Consider a positive smooth function $\varphi$ on $R_1$ satisfies
		$$||\varphi||_{p_0,\ R_{\sigma'}}\le \{C(\sigma-\sigma')^{-\gamma}\bar{\mu}^{-1}(R_1)\}^{1/p-1/p_0}||\varphi||_{p,\ R_{\sigma}}$$
		and 
		$$\bar{\mu}(\{z\in R_1,\ log\varphi>\lambda\})\le C\bar{\mu}(R_1)\lambda^{-1}$$
		for $\sigma,\ \sigma',\ p,\ \lambda$ satisfying $1/2\le\delta<\sigma'\le\sigma\le 1$, $0<p\le p_1<p_0$ and $\lambda>0$, then we have 
		$$||\varphi||_{p_0,\ R_{\delta}}\le (\bar{\mu}(R_1))^{1/p_0}e^{C_8(1+C^3)},$$
		where $C_8=C_8(\gamma,\ \delta,\ min\{1/p_1-1/p_0\}).$
	\end{lemma}
	It is easy to check that a positive solution of $\phi-$heat equation on $M^n$ satisfies all conditions outlined in the previous lemma, then we can finish the proof of the parabolic mean value inequality.
	\begin{theorem}\label{pbhnk}
		Let $(M^n, g, \mu)$ be a weighted Riemannian manifold under lower $N-Ricci\ curvature$ bound with $\eps-range$. Assume $\phi$ satisfies 
		$$0<a\le e^{\frac{2(1-\eps)\phi(x)}{n-1}}\le b.$$ Denote 
		$$Q_{+}=B_{x_0}( \delta R)\times[t_0-\eps R^2,\ t_0],\ Q_{-}=B_{x_0}(\delta R)\times[t_0-\varsigma R^2,\ t_0-\rho R^2],$$
		where $0<\eps<\rho<\varsigma\le 1$ and $R^2\le t_0$. Then, for any positive solution $u$ of $\phi-$heat equation, the following inequality holds 
		\begin{equation}
			\underset{ Q_{-}}{sup}\ u\le e^{2C_8(1+C_{9}^3)}\underset{ Q_{+}}{inf}\ u,
		\end{equation}
		where $C_8=c_{21}(\nu,\ \eps,\ \varsigma,\ \rho)$ and $C_{9}=2E_{5}exp(D_5\sqrt{K_{\eps}}R)$.
	\end{theorem}
	
	\begin{proof}
		Denote 
		$$Q_{+,\ \sigma}=B_{x_0}(\sigma R)\times [t_0-l_1(\sigma)\eps R^2,\ t_0]$$ 
		and 
		$$Q_{-,\ \sigma}=B_{x_0}(\sigma R)\times [t_0-l_2(\sigma)\varsigma R^2,\ t_0-\rho R^2],$$ 
		where $l_1$ and $l_2$ are two linear function satisfy 
		$$l_1(\delta)=l_2(1)\varsigma=l_2(\delta)=1,\  l_1(\delta)\eps=\rho.$$
		With mean value inequality \eqref{pbmvp} and Lemma \ref{wkLp}, it is easy to check that $f=e^{k}u$ satisies the condition in Lemma
		\ref{interpolation} for $(x,s)\in Q_{-,\ \sigma}$, where $k$ is same as that in Lemma \ref{wkLp}. Then we can obtain that 
		$$e^{k}||u||_{p_0,\ Q_{-},\frac{1+\delta}{2}}\le (\bar{\mu}(R_1))^{1/p_0}e^{C_8(1+C_{9}^3)}.$$
        By proposition \ref{pbmvp} ,
        \begin{equation}     
        \begin{aligned}
        e^{k}\underset{ Q_{-}}{sup}\ u &\le e^k\left(\frac{E_2 \exp\left(D_2 \sqrt{K_{\eps}(q,10R)} R\right)}{(1-\delta)^{2+\nu}R^{2}V_x(R)}\right)^{\frac{1}{p_0}}(\bar{\mu}(R_1))^{1/p_0}e^{C_8(1+C_{9}^3)}\\
        &=\left(\frac{E_2 \exp\left(D_2 \sqrt{K_{\eps}(q,10R)} R\right)}{(1-\delta)^{2+\nu}}\right)^{\frac{1}{p_0}}e^{C_8(1+C_{9}^3)+k}
         \end{aligned}
        \end{equation}
		Following same argument we can prove that
		$$e^{-k}\underset{ Q_{+}}{sup}\ u^{-1}\le e^{C_8(1+C_{9}^3)}.$$
		Finally, with last two inequality, following parabolic Harnack inequality holds.
		$$\underset{ Q_{-}}{sup}\ u\le C_{10}e^{2C_8(1+C_{9}^3)}\underset{ Q_{+}}{inf}\ u,$$
		where $C_8=c_{21}(\nu,\ \eps,\ \varsigma,\ \rho)$ , $C_{9}=2E_{5}exp(D_5\sqrt{K_{\eps}}R)$ and $C_{10}=\left(\frac{E_2 \exp\left(D_2 \sqrt{K_{\eps}(q,10R)} R\right)}{(1-\delta)^{2+\nu}}\right)^{\frac{1}{p_0}}$.
	\end{proof}
	
	\section{Upper and Lower Bounds of $\phi-$Heat Kernel}
	In this section, following the same argument in \cite{sal1}, we will derive the Gaussian upper bound of the $\phi-$heat kernel. Additionally, by applying the parabolic Harnack inequality established in the previous section, we will derive the lower bound of the $\phi-$heat kernel.
	
	To get the Gaussian upper bound of the heat kernel, let us first recall Davies’ double
	integral estimate of the metric measure space, which was proved in \cite{WW1}.
	\begin{lemma}\label{ddie}
		Let $(M^n, g, \mu)$ be a weighted Riemannian manifold and $\mu(M^n)\ge 0$ the bottom eigenvalue of $\phi-$Laplacian. Then, the following integral estimate holds
		\begin{equation}
			\int_{B_1} \int_{B_2} H_{\phi}(x, y, t) \, d\mu(y) \, d\mu(x) \leq \text{V}(B_1)^{\frac{1}{2}} \text{V}(B_2)^{\frac{1}{2}} e^{-\frac{ d^2(B_1, B_2)}{4t} - \mu_1(M) t},
		\end{equation}
		where $B_{1}$ and $B_2$ are two bounded sets and $d(B_1, B_2)$ denotes the distance between $B_{1}$ and $B_2$.
	\end{lemma}
	With this lemma, we can start the proof of main Theorem of this section

	\begin{proof}[(proof of the upper bound of Theorem \ref{hkestm}):]
		Fixing $x, y\in B_{o}(R/2)$,assume $R^{2}/4>t\ge r_{2}^{2}$. Denote $B_2=B_y(r_2)\subset B_o(R)$ and $Q_{\delta}=B_y(\delta r_2)\times (t-\delta r_{2}^2,\ t)$ with $0<\delta<1/4$. Then with parabolic mean value inequality \eqref{pbmvp}, following inequality holds
		\begin{equation}
			\begin{aligned}
				H_{\phi}(x,y,t)\le\underset{(z,s)\in Q_{\delta}}{sup}H_{\phi}(x,z,s)&\le \frac{E_2 \exp\left(D_2 \sqrt{K_{\eps}(q,10r_2)} r_2\right)}{(1-\delta)^{2+\nu}R^{2}\mu(B_2)}\int_{Q_\delta}H_{\phi}(x,z,s)d\mu(z) dt\\
				&\le\frac{E_2 \exp\left(D_2 \sqrt{K_{\eps}(q,10r_2)} r_2\right)}{(1-\delta)^{2+\nu}R^{2}\mu(B_2)}\int_{t-r_{2}^{2}/4}^{t}\int_{B_2}H_{\phi}(x,z,s)d\mu(z) dt\\
				&\le\frac{E_2 \exp\left(D_2 \sqrt{K_{\eps}(q,10r_2)} r_2\right)}{(1-\delta)^{2+\nu}4\mu(B_2)}\int_{B_2}H_{\phi}(x,z,s')d\mu(z),
			\end{aligned}
		\end{equation}
		for some $s'\in [\frac{3t}{4},\ t]$. In addition, $v(x,s)=\int_{B_2}H(x,z,s)d\mu(z)$ is also a positive solution, which means following inequality holds
		\begin{equation}
			\begin{aligned}
				\underset{(x,s)\in Q^{'}_{\delta}}{sup}\int_{B_2}H_{\phi}(x,z,s)d\mu(z)&\le \frac{E_2 \exp\left(D_2 \sqrt{K_{\eps}(q,10Rr_1)} r_1\right)}{(1-\delta)^{2+\nu}R^{2}\mu(B_1)}\int_{Q^{'}_{\delta}}\int_{B_2}H_{\phi}(x,z,s) d\mu(z)d\mu(x) dt\\
				&\le\frac{E_2 \exp\left(D_2 \sqrt{K_{\eps}(q,10r_1)} r_1\right)}{(1-\delta)^{2+\nu}4\mu(B_1)}\int_{B_1}\int_{B_2}H_{\phi}(x,z,s'')d\mu(z)d\mu(x)
			\end{aligned}
		\end{equation}
		for some \( s'' \in \left[ \frac{3t}{4}, t \right] \), where $( Q'_{\delta} = B_x(\delta r_1) \times \left(t - \delta r_{2}^2, t \right)$ and $B_1=B_x(r_1)\in B_o(R)$.
		
		Now letting $r_1=r_2=\sqrt{t}$, combining previous two inequalities, we can prove that
		\begin{equation}
			H_{\phi}(x,y,t)\le\frac{E_{2}^{2} \exp\left(2D_2 \sqrt{K_{\eps}(q,10\sqrt{t})} \sqrt{t}\right)}{16(1-\delta)^{2+\nu}\mu(B_1)\mu(B_2)}\int_{B_1}\int_{B_2}H_{\phi}(w,z,s'')d\mu(z)d\mu(w).
		\end{equation}
		Applying Lemma \ref{ddie}, we can derive following inequality
		\begin{equation}
			H_{\phi}(x,y,t)\le\frac{E_{2}^{2} \exp\left(2D_2 \sqrt{K_{\eps}(q,10\sqrt{t})} \sqrt{t}\right)}{16(1-\delta)^{4+2\nu}\mu(B_1)^{1/2}\mu(B_2)^{1/2}}exp(-\frac{3}{4}\mu_{1}t-\frac{d^2(B_1,B_2)}{4t}).
		\end{equation}
		Notice that if $d(x,y)\le 2\sqrt{t}\le R$, then 
		$$-d(B_1,\ B_2)=0\le 1-\frac{d^2(x,y)}{4t},$$
		and if $d(x,y)>2\sqrt{t}$, then $d(B_1,B_2)=d(x,y)-2\sqrt{t}$, which means 
		$$\frac{-d(B_1,\ B_2)}{4t}\le 1-\frac{d^2(x,y)}{4(1+\eps)t}+\frac{1}{\eps}.$$
		Then we can obtain that 
		$$exp(-\frac{3}{4}\mu_{1}t-\frac{d^2(B_1,B_2)}{4t})\le exp(1-\frac{d^2(x,y)}{4(1+\eps)t}+\frac{1}{\eps}),$$
		which means 
		\begin{equation}\label{uphk}
			H_{\phi}(x,y,t)\le\frac{C(\eps)E_{2}^{'} \exp\left(2D_2 \sqrt{K_{\eps}(q,10\sqrt{t})} \sqrt{t}\right)}{\mu(B_1)^{1/2}\mu(B_2)^{1/2}}\exp(-\frac{d^2(x,y)}{4(1+\eps)t}),
		\end{equation}
		where $\lim_{\eps\to0}C(\eps)=+\infty$.

    If $x,y\in M^n$, by letting $R\to\infty$, previous estimate holds for $\forall t>0$.
	\end{proof}
	Moreover, along with Lemma \ref{col2}, we can transform previous upper bound to the form only depend on the measure of the geodesic ball with center $x$.
	
	\begin{corollary}\label{ubhk2}
		Let $(M^n, g, \mu)$ be a weighted Riemannian manifold under lower $N-Ricci\ curvature$ bound with $\eps-range$. Assume $\phi$ satisfies 
		$$0<a\le e^{\frac{2(1-\eps)\phi(x)}{n-1}}\le b.$$ Then for any $x,y\in B_o(R/2)$, the $\phi-$heat kernal satisfies 
		\begin{equation}
			H_{\phi}(x,y,t)\le\frac{C(\eps)E_{6}^{2} (1+\frac{d(x,y)}{\sqrt{t}})^{(1+c)/2c}\exp\left(2D_2 \sqrt{K_{\eps}(q,10R)} R\right)}{V_x(R/2)}\exp(\sqrt{\frac{K_1}{c}}\frac{3R}{4a}-\frac{d^2(x,y)}{4(1+\eps)t}),
		\end{equation}
		where $E_6=(\frac{b}{a})^{(1+2c)/4c}E_2^{'}$.
	\end{corollary}
	Therefore, we already prove that the parabolic Harnack inequality in Theorem \eqref{pbhnk}  hold, we can apply the same argument as in \cite{sal2} to derive a Li-Yau type Harnack inequality. This result will be used to obtain the the lower bound of $\phi-$heat kernel. 
    \begin{proposition}\label{prop:LY_Harnack}
    Under the assumptions of Theorem \ref{pbhnk}, let $u>0$ be a solution to the $\phi$-heat equation on $B_o(R)\times (0,T)$. For any $0 < s < t < T$ and points $x,y \in B_o(R)$, the following Li-Yau type Harnack inequality holds:
    \begin{equation}\label{LYhnk}
        \ln\left(\frac{u(x, s)}{u(y, t)}\right) \leq C'_{10}\left( (t-s)\sqrt{D'_5 K_\epsilon} + \frac{d^2(x,y)}{t-s} + \frac{t-s}{R^2} + \frac{t-s}{s} \right),
    \end{equation}
    where $C'_{10}$ and $D'_5$ are constants depending on $\nu, n, a, b, c$, and $K$.
\end{proposition}

\begin{proof}
    Let $\gamma:[0,1]\to B_o(R)$ be a minimal geodesic connecting $x$ and $y$. Fix an integer $k\geq 1$ (to be specified later) and define parameters:
    \[
    \epsilon = \frac{1+\delta}{4}, \quad \varsigma = \frac{3+\delta}{4}, \quad \rho = \frac{3-\delta}{4} \quad\text{with}\quad \delta = \frac{4}{5}.
    \]
    Construct a chain of $k$ overlapping balls $\{B_\rho(x_i)\}_{i=0}^{k-1}$ along $\gamma$ such that:
    \begin{itemize}
        \item Centers $x_i \in \gamma$ satisfy $x_0 = x$, $x_k = y$, and $x_{i+1} \in B_\rho(x_i)$
        \item Time intervals $t_i = s + \rho^2 i$ for $0 \leq i \leq k$
    \end{itemize}
    The parameter $\rho$ is chosen to satisfy:
    \begin{enumerate}
        \item $\rho^2 = \frac{t-s}{k}$ ensuring $t_k = t$, which requires $k \geq \frac{d^2(x,y)}{t-s}$
        \item $\frac{9}{16}\rho^2 \leq s$ to guarantee $t_{i+1} - \frac{25}{16}\rho^2 \geq 0$ for all $i$
        \item $\rho \leq \frac{4}{5}R$ ensuring $B_{\frac{5}{4}\rho}(x_i) \subset B_{o}(2R)$
    \end{enumerate}
    
    For each cylinder $Q^i = B_{\frac{5}{4}\rho}(x_i) \times \left( t_{i+1}-\left(\frac{5}{4}\rho\right)^2, t_{i+1} \right)$, Theorem \ref{pbhnk} applies with:
    \[
    Q_-^i = B_\rho(x_i) \times \left( t_{i+1}-\frac{19}{20}\left(\frac{5}{4}\rho\right)^2, t_{i+1}-\frac{11}{20}\left(\frac{5}{4}\rho\right)^2 \right),
    \]
    \[
    Q_+^i = B_\rho(x_i) \times \left( t_{i+1}-\frac{9}{20}\left(\frac{5}{4}\rho\right)^2, t_{i+1} \right).
    \]
    Noting that $(x_i, t_i) \in \overline{Q_-^i}$ and $(x_{i+1}, t_{i+1}) \in \overline{Q_+^i}$, we obtain the iterative estimate:
    \[
    u(x_i, t_i) \leq e^{C'_8 \exp(D'_5\sqrt{K_\epsilon R})} u(x_{i+1}, t_{i+1}), \quad 0 \leq i \leq k-1.
    \]
    Chaining these inequalities yields:
    \[
    u(x,s) \leq \left( e^{C'_8 \exp(D'_5\sqrt{K_\epsilon R})} \right)^k u(y,t).
    \]
    
    To optimize $k$, choose the minimal integer satisfying:
    \[
    k \geq C\left( (t-s)\sqrt{D'_5 }K_\epsilon + \frac{d^2(x,y)}{t-s} + \frac{t-s}{R^2} + \frac{t-s}{s} \right).
    \]

Then
$$k \le C^{'}_9\left((t-s) \sqrt{D^{'}_5}K_\epsilon+\frac{d^{2}(x, y)}{t-s}+\frac{t-s}{R^{2}}+\frac{t-s}{s}+1\right) ,
$$

    This selection ensures compliance with conditions (C1)-(C3). Therefore, we conclude:
    \[
    \ln\left(\frac{u(x, s)}{u(y, t)}\right) \leq C'_{10}\left( (t-s)\sqrt{D'_5 }K_\epsilon + \frac{d^2(x,y)}{t-s} + \frac{t-s}{R^2} + \frac{t-s}{s} \right),
    \]
    completing the proof.
\end{proof}

	\begin{proof}[(proof of the lower bound of Theorem \ref{hkestm}):]
		Let $u(y,t)=H(x,y,t)$ with $x$ fixed and $s=t/2$. Then with Li-Yau-type Harnack inequality \eqref{LYhnk}, we can derive following estimate
		\begin{equation}\label{ineqhk1}
			H_{\phi}(x,y,t) \ge H_{\phi}(x,x,t/2) \times \exp \left( -C^{'}_{10} \left(  \sqrt{D^{'}_{5}}K_{\epsilon} \frac{t}{2} +   \frac{t}{2R^2} + 1 + \frac{2d^2(x,y)}{t}  \right) \right)
		\end{equation}
		for all $x,y\in B_{o}(R/2)$ and $0<t<\infty$. Denote $P_t$ is the heat semigroup of the weighted Laplacian, then for any $g(y)$ satisfies $0\le g(y)\le 1$, $g(y)\equiv 1$ in $B_{x}(\sqrt{t})$ and $g(y)\equiv 0$ in $M\setminus B_{x}(2\sqrt{t})$, we can construct a solution to the weighted heat equation with it as the initial value as
		\begin{equation}
			u(y,t)=
			\begin{cases}
				P_{t}g(y)& \text{if } t>0, \\
				g(y) & \text{if } t\le 0.
			\end{cases}
		\end{equation}
		Applying the local Harnack inequality, following estimate of the diagonal $\phi-$heat kernel holds
		\begin{equation}
			\begin{aligned}
				1=u(x,0)&\le e^{C^{'}_{11} \left(\sqrt{D^{'}_5}K_\epsilon t+1\right)}u(x,t/2)\\
				&=e^{C^{'}_{11} \left(\sqrt{D^{'}_5}K_\epsilon t+1\right)}\int_{M}H_{\phi}(x,y,t/2)g(y)d\mu(y)\\
				&\le e^{C^{'}_{11} \left(\sqrt{D^{'}_5}K_\epsilon t+1\right)}\int_{B_{x}(2\sqrt{t})}H_{\phi}(x,y,t/2)d\mu(y)\\
				&\le e^{C^{'}_{11} \left(\sqrt{D^{'}_5}K_\epsilon t+1\right)}(V_{x}(2\sqrt{t}))^{\frac{1}{2}}(H_{\phi}(x,x,t))^{\frac{1}{2}},\\
			\end{aligned}
		\end{equation}
		which means
		\begin{equation}
			H_{\phi}(x,x,t/2)\ge e^{-C^{'}_{11} \left(\sqrt{D^{'}_5}K_\epsilon t+2\right)}V_{x}^{-1}(\sqrt{2t})
		\end{equation}
		Since Lemma \ref{col1} implies
		\begin{equation}
			V_{x}(\sqrt{2t})\le (\frac{b}{a})^{(1+2c)/c}(2)^{(1+c)/c}exp(\sqrt{\frac{K_1}{c}}\frac{\sqrt{2t}}{a})V_{x}(\sqrt{t}),
		\end{equation}
		we then obtain
		\begin{equation}\label{ineqdghk}
			H_{\phi}(x,x,t/2)\ge E_{7}e^{-C^{'}_{11} \left(\sqrt{D^{'}_5}K_\epsilon t+2\right)-\frac{\sqrt{2t}}{a}}V_{x}^{-1}(\sqrt{t})
		\end{equation}
		for $0<\sqrt{t}<R/2$, where $E_7=(\frac{b}{a})^{(1+2c)/c}(2)^{(1+c)/c}$. Plugging inequality \eqref{ineqdghk} to \eqref{ineqhk1} yeilds
		\begin{equation}
			H_{\phi}(x,y,t) \ge C^{'}_{12} \exp \left( -C^{'}_{13} t - \frac{C^{'}_{14}d^2(x,y)}{ t} \right) V_{x}^{-1}(\sqrt{t}),
		\end{equation}
		where $C^{'}_{12},C^{'}_{13} , C^{'}_{14}$ all depend on $\nu ,n,a,b,c$ and $K$.
	\end{proof}
	\section{$L_{\phi}^{1}$ Liouville theorem}
	In this section, we will use $\phi-$heat kernel estimate to derive the $L_{\phi}^{1}$ Liouville theorem for $\phi-$subharmonic function on smooth metric measure space with lower $N-Ricci\ curvature$ in $\eps-range$. The reason we only study the $L_{\phi}^{1}$ case is that in \cite{SMA,Wu}, the authors prove that when $p>1$, without any assumption of the curvature, the $L_{\phi}^{p}$ Liouville theorem always holds.

	We start from a useful lemma in \cite{WW1}, which was first proved by A. Grigor’yan \cite{G1} on complete Riemannian manifold.
	\begin{lemma}
		Let $(M^n, g, \mu)$ be a smooth metric measure space, if 
		$$\int_{1}^{+\infty} \frac{R}{\log(V_{o}(R))}dR=+\infty$$
		for some point $o\in M^n$, then $(M^n, g, \mu)$ is stochastically complete, i.e.,
		$$\int_{M^n}H_{\phi}(x,y,t)d\mu(y)=1.$$
	\end{lemma}
	With Corollary \ref{col1}, we have 
	\begin{equation}
		V_x(R)\le c_{23}(b,K,n,\nu)(\frac{b}{a})^{(1+2c)/c}(R)^{(1+c)/c}exp(\sqrt{\frac{K_1}{c}}\frac{R}{a}).
	\end{equation}
	Then we can prove the stochastically completeness of smooth metric measure space with lower $N-Ricci\ curvature$ in $\eps-range$.
	\begin{corollary}\label{sccp}
		Let $(M^n, g, \mu)$ be a complete noncompact weighted Riemannian manifold with lower $N-Ricci\ curvature$ in $\eps-range$. Assume $\phi$ satisfies 
		$$0<a\le e^{\frac{2(1-\eps)\phi(x)}{n-1}}\le b,$$
		then $(M^n, g, \mu)$ is stochastically complete.
	\end{corollary}
	Now, we are ready to check the integration by parts formula by using the upper bound
	of the $\phi-$heat kernal.
	\begin{proposition}\label{prop1}
		Under the same assumption as Lemma \ref{sccp}, for any non-negative
		$L^{1}(\mu)-$integrable subharmonic function $h$, we have
		\begin{equation}
			\int_{M}\Delta_{\phi_{y}} H_{\phi}(x,y,t+s)h(y)d\mu(y)=	\int_{M} H_{\phi}(x,y,t+s)\Delta_{\phi_{y}}h(y)d\mu(y).
		\end{equation}
	\end{proposition}
	\begin{proof}
		By the Green formula on $B_{o}(R)$, we have
		\begin{equation}\label{Grf}
			\begin{aligned}
				&|\int_{B_{o}(R)}\Delta_{\phi_{y}} H_{\phi}(x,y,t)h(y)- H_{\phi}(x,y,t)\Delta_{\phi_{y}}h(y)d\mu_{y}|\\
				\le&|\int_{\pa B_{o}(R)}\pa_{r} H_{\phi}(x,y,t)h(y)- H_{\phi}(x,y,t)\pa_{r}h(y)d\mu_{\sigma,\ R}(y)|\\
				\le&\int_{\pa B_{o}(R)}|\nabla H_{\phi}(x,y,t)|h(y)d\mu_{\sigma,\ R}(y)+ \int_{\pa B_{o}(R)}H_{\phi}(x,y,t)|\nabla h(y)|d\mu_{\sigma,\ R}(y),\\
			\end{aligned}
		\end{equation}
		where $d\mu_{\sigma,\ R}(y)$ denotes the weighted area measure induced by $d\mu$ on $\pa B_{o}(R)$. We shall show that the above two boundary
		integrals vanish as $R\to\infty$.
		
		Step 1. Consider a large $R$ and denote $h(x,t)\equiv h(x)$ for any $(x,t)\in M^n\times [0,+\infty)$. Then the parabolic mean value inequality \eqref{pbmvp} deduces that 
		\begin{equation}\label{epmvp}
			\underset{B_{o}(R)}{sup}\{h(x)\}\le\frac{(2)^{2+\nu}E_2 \exp\left(2D_2 \sqrt{K_{\eps}(q,20R)} R\right)}{V_{o}(2R)}\int_{B_{o}(2R)}hd\mu dt.
		\end{equation}
		Let $\psi(y)=\psi(r(y))$ be a nonnagetive cut-off function satisfying $0\le\psi\le1$, $|\nabla\psi|\le\sqrt{3}$, $\psi(y)=1$ on $B_{o}(R+1)/B_{o}(R)$ and $\psi(y)=0$ on $B_{o}(R-1)\cup (M^{n}\setminus B_{o}(R+2))$. By the Cauchy-Schwarz inequality, we have
		\begin{equation}
			\begin{aligned}
				0 & \le \int_{M^n} h \Delta_\phi h \psi^2 d\mu \\
				& = -2 \int_{M^n} \psi h \langle \nabla h, \nabla\psi \rangle d\mu - \int_{M^n} |\nabla h|^2 \psi^2 d\mu \\
				& \le 2 \int_{M^n} |\nabla \psi|^{2} h^2 d\mu - \frac{1}{2} \int_{M^n} |\nabla h|^2 \psi^2 d\mu.
			\end{aligned}
		\end{equation}
		
		It then follows from \eqref{epmvp} that
		\begin{equation}
			\begin{aligned}
				\int_{B_{o}(R+1) \setminus B_{o}(R)} |\nabla h|^2 d\mu & \le 4 \int_{M^n} |\nabla \psi|^{2} h^2 d\mu \\
				& \le 12 \int_{B_{o}(R+2)} h^2 d\mu \\
				& \le 12 \underset{B_{o}(R+2)}{\sup} h \cdot ||h||_{L^{1}_{\mu}} \\
				& \le \frac{(2)^{2+\nu} E_2 \exp\left(D_2 \sqrt{K_{\eps}(q, 20(R+2))} (2R+4)\right)}{V_{o}(2R+4)} ||h||_{L^{1}_{\mu}}^2.
			\end{aligned}
		\end{equation}
		
		On the other hand, the Cauchy-Schwarz inequality implies that
		\begin{equation}
			\int_{B_{o}(R+1)\setminus B_{o}(R)} |\nabla h| d\mu\le(\int_{B_{o}(R+1)/B_{o}(R)} |\nabla h|^2 d\mu)^{1/2}(V_o(R+1)-V_o(R))^{1/2}
		\end{equation}
		Combining the above two inequalities, we have
		\begin{equation}
			\int_{B_{o}(R+1)\setminus B_{o}(R)} |\nabla h| d\mu \le (2)^{(2+\nu)/2} E_8 \exp\left(D_8 \sqrt{K_{\eps}(q, 20(R+2))} (R+2)\right) ||h||_{L^{1}_{\mu}}.
		\end{equation}
		
		Step 2. By letting $\eps=1/4$, the inequality \eqref{ubhk2} turns to
		\begin{equation}
			H_{\phi}(x,y,t)\le\frac{C_{12}(n,\nu,a,b) \exp\left(C_{13}(n,\nu,a,b,K) R\right)}{V_{x}(\sqrt{t})}(1+\frac{2R+1}{\sqrt{t}})^{(1+c)/2c}exp(-\frac{d^2(x,y)}{5t}),
		\end{equation}
		for all $x$, $y$ $\in B_{o}(R/4)$ and $0<t<R^2/16$.

		\begin{equation}
			\begin{aligned}
				J_1:&=\int_{B_{o}(R+1) \setminus B_{o}(R)} H_{\phi}(x,y,t) |\nabla h(y)| d\mu\\ & \le \underset{y \in B_{o}(R+1) \setminus B_{o}(R)}{\sup} H(x,y,t) \int_{B_{o}(R+1) \setminus B_{o}(R)} |\nabla h(y)| d\mu \\
				& \le \frac{C_{14}(n,\nu,a,b,||h||_{L^{1}_{\mu}}^{2}) \exp\left(C_{15}(n,\nu,a,b,K) R\right)}{V_{x}(\sqrt{t})}(1+\frac{d(x,y)}{\sqrt{t}})^{(1+c)/2c}exp(-\frac{(R-d(x,o))^2}{5t}).
			\end{aligned}
		\end{equation}
		Thus, for $T$ sufficiently small, all $t\in(0,T)$ and $d(x,o)\le R/8$, $J_1\to 0$ when $R\to\infty$.
		
		Step 3. We show that $\int_{B_{o}(R+1) \setminus B_{o}(R)} |\nabla H_{\phi}|(x,y,t)h(y) d\mu(y)\to 0$ as $R\to\infty$.  Frist, consider the integral
		
		\begin{equation}
			\begin{aligned}
				\int_{M} |\nabla H_{\phi}(x,y,t)|^2 \psi^2(y) \, d\mu &= -2 \int_{M} \langle H_{\phi}(x,y,t) \nabla \psi, \psi \nabla H_{\phi}(x,y,t) \rangle \, d\mu \\
				&\quad - \int_{M} H_{\phi}(x,y,t) \Delta_\phi H_{\phi}(x,y,t) \psi^2(y) \, d\mu \\
				&\le 2 \int_{M} |\nabla \psi|^2 H_{\phi}^{2}(x,y,t) \, d\mu + \frac{1}{2} \int_{M} |\nabla H_{\phi}(x,y,t)|^2 \psi^2(y) \, d\mu \\
				&\quad - \int_{M} H_{\phi}(x,y,t) \Delta_\phi H_{\phi}(x,y,t) \psi^2(y) \, d\mu,
			\end{aligned}
		\end{equation}
		which implies
		\begin{equation}\label{nblH}
			\begin{aligned}
				&\int_{B_{o}(R+1) \setminus B_{o}(R)} |\nabla H_{\phi}|^2 d\mu\\ &\le \int_{M^n} |\nabla H_{\phi}| ^2\psi^2 d\mu \\
				& \le 4 \int_{M} |\nabla \psi|^2 H_{\phi}^2-2\int_{M}\psi^2H\Delta_\phi H_{\phi} d\mu \\
				& \le 12 \int_{B_{o}(R+1) \setminus B_{o}(R)}  H_{\phi}^2 d\mu+2	\int_{B_{o}(R+1) \setminus B_{o}(R)} H_{\phi}|\Delta_\phi H_{\phi}| d\mu \\
				& \le 12 \int_{B_{o}(R+1) \setminus B_{o}(R)}  H_{\phi}^2 d\mu+2 (\int_{B_{o}(R+1) \setminus B_{o}(R)}  H_{\phi}^2 d\mu)^{1/2}(\int_{B_{o}(R+1) \setminus B_{o}(R)}  |\Delta_\phi H|^2 d\mu)^{1/2}.
			\end{aligned}
		\end{equation}
		We already prove that $(M^n, g, \mu)$ is stochastically complete, which means
		\begin{equation}
			\int_{M}H_{\phi}(x,y,t)d\mu(y)=1.
		\end{equation}
		Then with Corollary \ref{ubhk2}, we can obtain that
		\begin{equation}\label{ineqH2}
			\begin{aligned}
				&\int_{B_{o}(R+1) \setminus B_{o}(R)}  H_{\phi}^2 d\mu \le \underset{y\in B_{o}(R+1)\setminus B_{o}(R)}{\sup}H_{\phi}(x,y,t)\\
				&\le \frac{C_{12} \exp\left(C_{13} R\right)}{V_{x}(\sqrt{t})}(1+\frac{d(x,o)+R+2}{\sqrt{t}})^{(1+c)/2c}exp(-\frac{(R-1-d(x,o)^2)}{5t})\\
			\end{aligned}
		\end{equation}
		From (4.7) in \cite{WW2}, there exists a constant $C>0$ such that 
		\begin{equation}\label{ineqhk}
			\int_{M}(\Delta_\phi H_{\phi})^2d\mu\le \frac{C_{16}}{t^2}H_{\phi}(x,x,t).
		\end{equation}
		Combining \eqref{nblH}, \eqref{ineqH2} and \eqref{ineqhk}, we obtain
		\begin{equation}
			\begin{aligned}
				\int_{B_{o}(R+1) \setminus B_{o}(R)} |\nabla H_{\phi}|^2 d\mu\le &C_{17}e^{C_{18}R-\frac{(R-1-d(x,o))^2}{10t}}[V_{x}(\sqrt{t})^{-1}+V_{x}(\sqrt{t})^{-1/2}t^{-1}H^{1/2}(x,x,t)]^{1/2}\\
				&\times(1+\frac{d(x,o)+R+2}{\sqrt{t}})^{(1+c)/2c},
			\end{aligned}
		\end{equation}
		where $C_{17}$ and $C_{18}$ are some constants not depends on $R$. By the Cauchy-Schwarz inequality, we get
		\begin{equation}\label{nbhk3}
			\begin{aligned}
				&\int_{B_{o}(R+1) \setminus B_{o}(R)} |\nabla H_{\phi}| d\mu\le [V_{o}(R+1)-V_{o}(R)]^{1/2}(\int_{B_{o}(R+1) \setminus B_{o}(R)} |\nabla H_{\phi}|^2 d\mu)^{1/2}\\
				&\le V_{o}(R+1)^{1/2}C_{19}e^{C_{20}R-\frac{(R-1-d(x,o))^2}{20t}}[V_{x}(\sqrt{t})^{-1}+V_{x}(\sqrt{t})^{-1/2}t^{-1}H^{1/2}(x,x,t)]^{1/2}\\
				&\times(1+\frac{d(x,o)+R+2}{\sqrt{t}})^{(1+c)/4c},
			\end{aligned}
		\end{equation}
		Therefore, by \eqref{epmvp} and \eqref{nbhk3}, we have 
		\begin{equation}
			\begin{aligned}
				J_{2}:&= \int_{B_{o}(R+1) \setminus B_{o}(R)} |\nabla H_{\phi}|(x,y,t)h(y) d\mu(y)\\
				&\le \underset{y\in B_{o}(R+2)\setminus B_{o}(R+1)}{\sup}h(y)\int_{B_{o}(R+1) \setminus B_{o}(R)} |\nabla H_{\phi}|(x,y,t) d\mu(y)\\
				&\le C_{21}e^{C_{22}R-\frac{(R-1-d(x,o))^2}{20t}}||h||_{L^{1}(\mu)}[V_{x}(\sqrt{t})^{-1}+V_{x}(\sqrt{t})^{-1/2}t^{-1}H^{1/2}(x,x,t)]^{1/2}\\
				&\times (1+\frac{d(x,o)+R+2}{\sqrt{t}})^{(1+c)/4c}.
			\end{aligned}
		\end{equation}
		Similar to the case of $J_1$, for $T$ sufficiently small, all $t\in(0,T)$ and $d(x,o)\le R/8$, $J_2\to 0$ when $R\to\infty$.
		
		Step 4.  By the mean value theorem, for all $R>0$ there exists $\bar{R}\in (R,R+1)$ such that
		\begin{equation}
			\begin{aligned}
				J:&=\int_{\pa B_{o}(\bar{R})}|\nabla H_{\phi}(x,y,t)|h(y)d\mu_{\sigma,\ R}(y)+ \int_{\pa B_{o}(\bar{R})}H_{\phi}(x,y,t)|\nabla h(y)|d\mu_{\sigma,\ R}(y)\\
				&=\int_{B_{o}(R+1) \setminus B_{o}(R)} |\nabla H_{\phi}|(x,y,t)h(y) d\mu(y)+\int_{B_{o}(R+1) \setminus B_{o}(R)} H_{\phi}(x,y,t) |\nabla h(y)| d\mu\\
				&=J_{2}+J_{1}.
			\end{aligned}
		\end{equation}
		From step 2 and step 3, we know that for sufficiently small $t$, $J\to 0$ when $R\to\infty$, which means this Proposition holds for sufficiently small $t$.
		
		Step 5. For all $t\in(0,T)$ and $s\in(0,+\infty)$, using the semigroup property of the weighted heat kernel, we have
		\begin{equation}
			\begin{aligned}
				\int_{M}\Delta_{\phi_{y}} H_{\phi}(x,y,t+s)h(y)d\mu(y)&=\int_{M}\int_{M} H_{\phi}(x,z,t)\Delta_{\phi_{y}}H_{\phi}(z,y,s)d\mu(z)\ h(y)d\mu(y)\\
				&=\int_{M}\int_{M} \Delta_{\phi_{y}}H_{\phi}(z,y,s)h(y)d\mu(y)\ H_{\phi}(x,z,t)d\mu(z)\\
				&=\int_{M}\int_{M} H_{\phi}(z,y,s)\Delta_{\phi_{y}}h(y)d\mu(y)\ H_{\phi}(x,z,t)d\mu(z)\\
				&=\int_{M} H_{\phi}(x,y,t+s)\Delta_{\phi_{y}}h(y)d\mu(y).
			\end{aligned}
		\end{equation}
		Then for any $t>0$, \eqref{Grf} always holds, which finish the proof of this proposition.
	\end{proof}
	
	Next we prove the $L_{\phi}^{1}$ Liouville theorem,
	
	\begin{proof}[(Proof of Theorem\ref{liouville}):]
		Let $h(x)$ be a nonnegative $\phi-$subharmonic function, $L_{\phi}^{1}(\mu)$ integrable on $M^n$. Denote a space-time function
		$$h(x,t)=\int_{M^n}h(y)H_{\phi}(x,y,t)d\mu(y)$$
		with initial data $h(x,0)=h(x)$. From previous Proposition, we conclude that
		\begin{equation}
			\begin{aligned}
				\pa_t h(x,t)&=\int_{M}\Delta_{\phi_{y}} H_{\phi}(x,y,t)h(y)d\mu(y)\\
				&=\int_{M} H_{\phi}(x,y,t)\Delta_{\phi_{y}}h(y)d\mu(y)\\
				&\ge 0.\\
			\end{aligned}
		\end{equation}
		that is, $h(x, t)$ is increasing in $t$. In addition, we already prove that 
		$$\int_{M}H(x,y,t)d\mu(y)=1$$
		for any $x\in M^n$ and $t>0$. So we have
		\begin{equation}
			\int_{M}h(x,t)d\mu(x)=\int_{M}\int_{M^n}h(y)H_{\phi}(x,y,t)d\mu(x)d\mu(y)=	\int_{M}h(y)d\mu(y).
		\end{equation}
		Since $h(x, t)$ is increasing in $t$, so $h(x, t) = h(x)$ and hence $h(x)$ is a nonnegative
		$\phi -$harmonic function, i.e. $\Delta_\phi h=0$.
		
		On the other hand, for any positive constant $A$, let us define a new function $u(x) =
		min\{h(x), A\}$. Then $u$ satisfies
		$$0\le u(x)\le h(x),\  |\nabla u|\le|\nabla h|\  \text{and}\  \Delta_\phi u\le 0.$$
		By the same argument, we have that $\Delta_\phi u=0$.
		
		By the regularity theory of  $\phi-$harmonic functions, this is impossible unless $h = u$ or
		$u = A$. Since $A$ is arbitrary and $h$ is nonnegative, so $h$ must be identically constant. The
		theorem then follows from the fact that the absolute value of a $\phi-$harmonic function
		is a nonnegative $\phi-$subharmonic function.
	\end{proof}
	\begin{remark}
		With this Theorem, we remove the restriction of $\frac{b}{a}$ in \cite{Fujitani2} and derive the general $L_{\phi}^{1}$ Liouville theorem of weighted Riemannian manifolds under lower \(N\)-Ricci curvature bound with \(\varepsilon\)-range.
	\end{remark}
	With the $L^{1}(\phi)-$Liouville property, one can prove the uniqueness of  $L^{1}(\phi)-$ solution of the weighted heat equation.
	\begin{theorem}\label{L1uniq}
		Let $(M^n, g, \mu)$ be a complete noncompact weighted Riemannian manifold with a lower \(N\)-Ricci curvature bound and \(\varepsilon\)-range. Assume \(\phi\) satisfies 
		$$0 < a \le e^{\frac{2(1-\eps)\phi(x)}{n-1}} \le b.$$ 
		If $h(x, t)$ is a non-negative function defined on $M\times[0, +\infty)$ satisfying
		$$(\Delta_\phi-\pa_t)h(x,t)\ge 0,\   \int_{M}h(x,t)d\mu(x)<\infty$$ 
		for all $t > 0$, and $\lim_{t\to0}\int_{M}u(x,t)d\mu(x)=0$, then $u(x,t)\equiv 0$.
		
		In particular, any $L^{1}_{\phi}-$solution of the heat equation is uniquely determined by its initial data in $L^{1}(\mu)$.
	\end{theorem}
	
	\begin{proof}
		Let $h(x,t)\in L^{1}_{\mu}$ be a non-negative function satisfying the assumptions in 
		Theorem. For $\eps>0$, we define a space-time function
		$$h_{\eps}(x,t)=\int_{M}H(x,y,t)h(y,\eps)d\mu(y)$$
		and
		$$F_{\eps}(x,t)=max\{0,h(x,t+\eps)-h_{\eps}(x,t)\}.$$
		Then $F_{\eps}(x,t)$ is non-negative and satisfies
		$$\lim_{t\to0}F_{\eps}(x,t)=0,\ (\Delta_\phi-\pa_t)F_{\eps}(x,t)\ge0.$$
		Fixed $T>0$, let $u(x)=\int_{0}^{T}F_{\eps}(x,t)dt$, which implies
		\begin{equation}\label{Deux}
			\Delta_\phi u(x)=\int_{0}^{T}\Delta_\phi F_{\eps}(x,t)dt\ge\int_{0}^{T}\pa_t F_{\eps}(x,t)dt=F_{\eps}(x,T)\ge 0,
		\end{equation}
		and
		\begin{equation}
			\begin{aligned}
				\int_{M}u(x)d\mu(x)&=\int_{0}^{T}\int_{M} F_{\eps}(x,t)d\mu(x)dt\le\int_{0}^{T}\int_{M} |h(x,t+\eps)-h_{\eps}(x,t)|d\mu(x)dt\\
				&\le \int_{0}^{T}\int_{M} h(x,t+\eps)d\mu(x)dt+\int_{0}^{T}\int_{M} h_{\eps}(x,t)d\mu(x)dt< \infty.
			\end{aligned}
		\end{equation}
		where the first term on the right hand is finite from our assumption, and the second term
		is finite because the semigroup is contractive on $L^{1}(\mu)$. Therefore, $u(x)$ is a non-negative $L^{1}(\mu)$ integrable subharmonic function. By the Liouville Theorem we already prove, $u(x)$ must be constant. Combining with \eqref{Deux} we have $F_{\eps}(x,t)=0$ for all $x\in M^n$ and $t>0$, which implies
		\begin{equation}\label{large}
			h(x,t+\eps)\le h_{\eps}(x,t).
		\end{equation}
		Next we estimate the function $h_{\eps}(x,t)$. Applying the upper bound estimate \eqref{ubhk2} of the weighted heat kernel $H_{\phi}(x, y, t)$ and letting $\eps= 1/4,\ R = 2d(x,y)+1$, we have
		\begin{equation}
			h_{\eps}(x,t)\le \frac{C_{24}}{V_{x}(\sqrt{t})}\int_{M}e^{C_{25}d(x,y)-\frac{d^{2}(x,y)}{5t}}(1+\frac{d(x,y)}{\sqrt{t}})^{(1+c)/2c}h(y,\eps)d\mu(y)
		\end{equation}
		For sufficiently small values of $t > 0$, the right-hand side can be estimated by
		\begin{equation}
			h_{\eps}(x,t)\le \frac{C_{25}}{V_{x}(\sqrt{t})}\int_{M}h(y,\eps)d\mu(y)
		\end{equation}
		Hence as $\eps\to 0$, $h_{\eps}(x,t)\to 0$ since $\int_{M}h(y,\eps)d\mu(y)\to 0$.However, by the semigroup property,
		$h_{\eps}(x,t)\to 0$ for all $x\in M$ and $t > 0$. Combining with \eqref{large} we get $h(x, t)\le 0$. Therefore $h(x,t)\equiv 0$.
		
		To prove that any $L^{1}_{\phi}-$ solution of the heat equation is uniquely determined by its initial data in $L^{1}(\mu)$, we suppose that $u_{1}(x, t),\ u_{2}(x, t)$ are two $L^{1}(\mu)-$integrable solutions of the weighted heat equation $(\Delta_\phi-\pa_t)u(x,t)=0$ with the initial data $u(x,0)\in L^1(\mu)$. Applying this above result to $v(x,t)=|u_{1}(x, t)-u_{2}(x, t)|$, we see that $v(x,t)\equiv 0$, which finish the proof of this Theorem.
	\end{proof}
	\section{Eigenvalue estimates}
	In this section, we derive lower bound estimations of eigenvalues of the weighted Laplacian operator $\Delta_\phi$ on compact weighted Riemannian manifolds by using the upper bound of the $\phi-$heat kernel and an argument of Li-Yau\cite{LiYau1986}.
	
	First, let's assume that $N-Ricci$ is nonnegative
	\begin{theorem}\label{lbeg}
		Let $(M^n, g, \mu)$ be a compact weighted Riemannian manifold under lower $N-Ricci\ curvature$ in $\eps-range$. Assume $\phi$ satisfies 
		$$0<a\le e^{\frac{2(1-\eps)\phi(x)}{n-1}}\le b ,$$
		 and let $K=0$. Denote the eigenvalues of $\Delta_\phi$ by $0=\lambda_{0}<\lambda_{1}\le \lambda_{2}\le\ldots$. Then, the following lower bound estimate of $\lambda_{k}$ holds
		\begin{equation}
			\lambda_{k}\ge \frac{C_{26}(k+1)^{2c/(c+1)}}{d^{2}},
		\end{equation}
		where $C_{26}=c_{25}(n,\nu)(\frac{b}{a})^{c_{26}(n,\nu)}$.
	\end{theorem} 
	\begin{proof}
		Since $Ric_N^{\phi}\ge 0$, from Corollary \ref{ubhk2}, we know that the upper bound of the diagonal $\phi-$heat kernel satisfies
		\begin{equation}\label{ubdghk}
			H(x,x,t)\le \frac{c_{23}(n,\nu)(\frac{b}{a})^{c_{24}(n,\nu)}}{V_x(\sqrt{t})}.
		\end{equation}
		Notice that $\phi-$ heat kernel can be written as 
		$$H_{\phi}(x,y,t)=\sum_{i=0}^{\infty}e^{-\lambda_{i}t}\varphi_{i}(x)\varphi_{i}(y),$$
		where $\varphi_{i}$ is the eigenfunction of $\Delta_\phi$ corresponding to $\lambda_{i}$, $||\varphi_i||_{L^{2}(\mu)}=1$. By Corollary \ref{col1}, we know that
		$$\frac{V_{x}(d)}{V_{x}(\sqrt{t})}\le (\frac{b}{a})^{(1+2c)/c}(\frac{d}{\sqrt{t}})^{(1+c)/c},$$
		Taking the weighted integral on both sides of \eqref{ubdghk}, we conclude that
		\begin{equation}
			\sum_{i=0}^{\infty} e^{-\lambda_{i}t} \leq \int_{M} \frac{c_{23}(n,\nu) \left( \frac{b}{a} \right)^{c_{24}(n,\nu)}}{V_x(\sqrt{t})} \, d\mu(x) \leq c_{23}(n,\nu) \left( \frac{b}{a} \right)^{c_{24}(n,\nu)} \int_{M} p(t) \, d\mu(x).
		\end{equation}
		where 
		\begin{equation}
			p(t)=
			\begin{cases}
				(\frac{b}{a})^{(1+2c)/c}(\frac{d}{\sqrt{t}})^{(1+c)/c}V_{x}(d)^{-1}& \text{if } \sqrt{t}\le d, \\
				(\frac{b}{a})^{(1+2c)/c}(\frac{d}{\sqrt{t}})^{(1+c)/c}V(M)^{-1} & \text{if } \sqrt{t}\ge d,
			\end{cases}
		\end{equation} 
		which implies that $(k+1)e^{(-\lambda_{k}t)}\le c_{23}(n,\nu) \left( \frac{b}{a} \right)^{c_{24}(n,\nu)}q(t)$ for any $t>0$, that is 
		\begin{equation}\label{upbdeg}
			e^{(\lambda_{k}t)}q(t)\ge (k+1)c^{-1}_{23}(n,\nu) \left( \frac{b}{a} \right)^{-c_{24}(n,\nu)},
		\end{equation}
		where 
		\begin{equation}
			q(t)=
			\begin{cases}
				(\frac{b}{a})^{(1+2c)/c}(\frac{d}{\sqrt{t}})^{(1+c)/c}& \text{if } \sqrt{t}\le d, \\
				(\frac{b}{a})^{(1+2c)/c}(\frac{d}{\sqrt{t}})^{(1+c)/c} & \text{if } \sqrt{t}\ge d.
			\end{cases}
		\end{equation} 
		It is easy to check that when $t_{0}=\frac{n}{2\lambda_{k}}$. Plugging to \eqref{upbdeg}
		we get the lower bound for $\lambda_{k}$.
	\end{proof}
	When $K\neq 0$, following the same argument as that when $K=0$, we can still derive lower bound estimations of eigenvalues of the weighted Laplacian operator $\Delta_\phi$ as
	\begin{theorem}\label{lbeg2}
		Let $(M^n, g, \mu)$ be a compact weighted Riemannian manifold under lower $N-Ricci\ curvature$ in $\eps-range$. Assume $\phi$ satisfies 
		$$0<a\le e^{\frac{2(1-\eps)\phi(x)}{n-1}}\le b.$$
		Let $0=\lambda_{0}<\lambda_{1}\le \lambda_{2}\le...$ be eigenvalues of $\Delta_\phi$. Then following lower bound estimate of $\lambda_{k}$ holds
		\begin{equation}
			\lambda_{k}\ge \frac{C_{27}(a,b,n,\nu)}{d^{2}}(\frac{k+1}{exp(((\frac{b}{a})^{c_{24}(n,v)})\sqrt{K_{\eps}(M^n)}d)})^{2c/(c+1)}.
		\end{equation}
	\end{theorem} 
	
	\section{Li-Yau's Gradient estimate for $\phi-$Heat equation}
In the previous section, we established the Gaussian upper bound for the $\phi$-heat kernel. Building on this result, we derive a \textit{Li-Yau-type} gradient estimate for the $\phi$-heat equation by adapting the techniques from \cite{ZhangZhu2018}, under appropriate $L^p$-norm assumptions on the weighted gradient $\||\nabla\phi|\|_{L^p(\mu)}$ where $p\geq n$.

Denote $Q=J(x,t)\left(\frac{|\nabla u|^2}{u^2}\right)-\alpha\left(\frac{\pa_t u}{u}\right)$ where $J(x,t)$ is some smooth function on $M$. Let $\omega=ln u$, it is clear that
\begin{equation}
	\pa_t \omega=\Delta_\phi\omega+|\nabla\omega|^2.
\end{equation}
Then by simply caculate, we have
\begin{equation}
	\begin{aligned}
		(\Delta_\phi - \partial_t)Q 
		&= 2J \Big( \operatorname{Ric}_\phi^N(\nabla w, \nabla w) + |\operatorname{Hess} w|^2 \Big) 
		+ 2J \frac{\langle \nabla \phi, \nabla w \rangle^2}{N - n} \\
		&\quad + 2\langle \nabla w, \nabla J \rangle |\nabla w|^2 
		- 2\alpha \langle \nabla w, \nabla w_t \rangle 
		- 2\langle \nabla w, \nabla Q \rangle \\
		&\quad + (\Delta_\phi J)|\nabla w|^2 
		+ 2 \langle \nabla J, \nabla|\nabla w|^2 \rangle 
		- \partial_t J|\nabla w|^2 \\
		&\quad + \alpha \partial_t(|\nabla w|^2)
	\end{aligned}
\end{equation}
Using the Cauchy-Schwarz inequality
\begin{equation}
	\begin{aligned}
		2\langle\nabla J, \nabla(|\nabla w|^2)\rangle &= 4\langle\nabla J, \text{Hess}(\nabla w, \cdot)\rangle \\
		&\geq -4|\nabla J| |\nabla w| \text{Hess} w| \\
		&\geq -\frac{4|\nabla J|^2|\nabla w|^2}{\delta J} - \delta J |\text{Hess} w|^2
	\end{aligned}
\end{equation}
and
\begin{equation}
	2\langle\nabla J,\nabla\omega \rangle\ge -\delta J |\nabla\omega|^2-\frac{|\nabla J|^2}{\delta J}.
\end{equation}

\begin{equation}
	\begin{aligned}
		|\operatorname{Hess} \omega|^2&\ge \frac{(\Delta\omega)^2}{n}\\
		&=\frac{(\Delta_{\phi}\omega+\langle \nabla \phi, \nabla w \rangle)^2}{n}\\&
		\ge \frac{(\Delta_{\phi}\omega)^{2}-2|\nabla \phi|^{2}|\nabla w|^2}{2n}.
	\end{aligned}
\end{equation}
Finally we can derive following estimate
\begin{equation}\label{Qestimate}
	\begin{aligned}
		(\Delta_\phi - \partial_t)Q 
		\geq{} & 2J \operatorname{Ric}_\phi^N(\nabla w, \nabla w) + (2 - \delta)J |\operatorname{Hess} w|^2 
		- \frac{5|\nabla J|^2}{\delta J}|\nabla w|^2 - \delta J|\nabla w|^4 \\
		& - 2\langle \nabla w, \nabla Q \rangle - 2\alpha\langle\nabla w, \nabla w_t\rangle 
		+ (\Delta_\phi J)|\nabla w|^2 - J_t|\nabla w|^2 \\
		& + \alpha \partial_t(|\nabla w|^2) + \frac{2J\langle \nabla\phi, \nabla w\rangle^2}{N - n} \\
		\geq{} & 2J \operatorname{Ric}_\phi^N(\nabla w, \nabla w) 
		+ (2 - \delta)J \left( \frac{(\Delta_\phi w)^2 - 2\langle\nabla \phi, \nabla w\rangle^2}{2n} \right) \\
		& + \frac{2J\langle\nabla\phi, \nabla w\rangle^2}{N - n} - \frac{5|\nabla J|^2}{\delta J}|\nabla w|^2 
		- \delta J|\nabla w|^4 \\
		& - 2\alpha\langle\nabla w, \nabla w_t\rangle - 2\langle\nabla w, \nabla Q\rangle 
		+ (\Delta_\phi J)|\nabla w|^2 - J_t|\nabla w|^2 + \alpha \partial_t(|\nabla w|^2) \\
		\geq{} & \left( \Delta_\phi J - \frac{5|\nabla J|^2}{\delta J} - J_t \right)|\nabla w|^2 
		+ 2J \Big( \operatorname{Ric}_\phi^N(\nabla w, \nabla w) \\
		& - C_{28}(n,N,\delta)|\nabla\phi|^2|\nabla w|^2 \Big) 
		- \delta J|\nabla w|^4 + \frac{(2 - \delta)J}{2n}(\Delta_\phi w)^2 \\
		& - 2\langle\nabla w, \nabla Q\rangle \\
		\geq{} & \left( \Delta_\phi J - \frac{5|\nabla J|^2}{\delta J} - J_t - 2JV \right)|\nabla w|^2 
		- \delta J|\nabla w|^4 \\
		& + \frac{(2 - \delta)J}{2n}(\Delta_\phi w)^2 - 2\langle\nabla w, \nabla Q\rangle
	\end{aligned}
\end{equation}
here $C_{28}(n,N,\delta)=\frac{2n+(2-\delta)(N-n)}{2n(N-n)}$ and $V(x)=max\{Ke^{\frac{4(\eps-1)\phi(x)}{n-1}}+C_{28}(n,N,\delta)|\nabla\phi|^2,0\}$.

\begin{lemma}
    Under the same setting as pervious caculates, for $\delta\in (0,2)$, $p> n$ and $V(x)=max\{Ke^{\frac{4(\eps-1)\phi(x)}{n-1}}+C_{28}(n,N,\delta)|\nabla\phi|^2,0\}$ , the problem

$\left\{
\begin{aligned}
&\Delta_\phi J - 2VJ - \frac{5}{\delta} \frac{|\nabla J|^2}{J} - \partial_t J = 0 \quad \text{on } M \times (0, \infty) \\
&J(\cdot, 0) = 1
\end{aligned}
\right.$
has a unique solution for $t \geq 0$ , which satisfies
$$\underline{J}(t) \leq J(x,t) \leq 1$$

where 
\begin{equation}
    \underline{J}(t) = 2^{-\frac{1}{\tau-1}} e^{-(\tau-1)^{\frac{n}{(2p-n)}} (4 C_{29} \hat{C}(t)^{\frac{1}{p}})^{\frac{2p}{2p-n}} t}
\end{equation} 
and
$$
\begin{aligned}
    &C_{29} = K \| e^{\frac{4(\varepsilon-1) \varphi(x)}{n-1}} \|_p + C_{28} \| |\nabla \varphi|^2 \|_p \\
&\hat{C}(t) = C(\varepsilon) E_2^{'} \exp \left(2 D_2 \sqrt{K_\varepsilon(q,10\sqrt{t})} \sqrt{t}\right).
\end{aligned}
$$

\end{lemma}
\begin{proof}
Let $\tau = \frac{5}{\delta}$ and define $w$ through the substitution $J = w^{-\frac{1}{\tau-1}}$. Under this transformation, the PDE for $J$ reduces to:

\small{\[
\frac{\tau}{(\tau-1)^2} w^{-\frac{2\tau-1}{\tau-1}} |\nabla w|^2 - \frac{1}{\tau-1} w^{-\frac{\tau}{\tau-1}} \Delta_\phi w - 2V w^{-\frac{1}{\tau-1}} - \frac{5}{\delta} \frac{1}{(\tau-1)^2} w^{-\frac{2\tau-1}{\tau-1}} |\nabla w|^2 = -\frac{1}{\tau-1} w^{-\frac{\tau}{\tau-1}} w_t
\]}

This leads to the following parabolic system for $w$:
\begin{equation}
	\left\{
	\begin{aligned}
		&\Delta_\phi w + 2(\tau-1)Vw = w_t \quad \text{on } M \times (0, \infty), \\
		&w(\cdot, 0) = 1.
	\end{aligned}
	\right.
\end{equation}
The existence of solutions follows from the Lax-Milgram theorem, while uniqueness is guaranteed by the Liouville theorem established in Section~5.

Applying Duhamel's principle, we reformulate (5) as an integral equation:
\begin{equation}
	w(x,t) = 1 + 2(\tau-1) \int_0^t \int_M H_\phi(x,y,t-s) V(y)w(y,s) \, d\mu(y) \, ds
\end{equation}
with the heat kernel estimate:

\begin{equation}
\begin{aligned}
	H_\phi(x,y,t) & \leq \frac{C(\varepsilon) E_2^2 \exp\left(2D_2 \sqrt{K_\varepsilon(q,10\sqrt{t})} \sqrt{t}\right)}{\sqrt{V_x(\sqrt{t}) V_y(\sqrt{t})}} \exp\left(-\frac{d^2(x,y)}{4(1+\varepsilon)t}\right)\\
    &=\hat{C}(t)\frac{1}{\sqrt{V_x(\sqrt{t}) V_y(\sqrt{t})}}\exp\left(-\frac{d^2(x,y)}{4(1+\varepsilon)t}\right).
    \end{aligned}
\end{equation}

To estimate $w$, define 
\[
m(t) = \sup_{M \times [0,t]} w(x,s),
\]
which satisfies:
\begin{equation}
	\begin{aligned}
		w(x,t) \leq 1 & + 2(\tau-1) \int_0^{t-\xi} \int_M H_\phi(x,y,t-s) V(y)m(s) \, d\mu(y) \, ds \\
		& + 2(\tau-1) \int_{t-\xi}^t \int_M H_\phi(x,y,t-s) V(y)m(s) \, d\mu(y) \, ds.
	\end{aligned}
\end{equation}

Employing H\"older's inequality with the heat kernel estimate (Theorem \ref{hkestm}) and volume comparison theorem (Theorem \ref{VCP}), we derive:
\begin{equation}
	\begin{aligned}
		\int_M H_\phi(x,y,t-s)V(y) \, d\mu(y) 
		&\leq \left( \int_M V^p(y) \, d\mu(y) \right)^{\frac{1}{p}} \left( \int_M H_\phi^{\frac{p}{p-1}}(x,y,t-s) \, d\mu(y) \right)^{\frac{p-1}{p}} \\
		&\leq \left[ K \left\| e^{\frac{4(\varepsilon-1)\phi}{N-1}} \right\|_p + C_{28} \|\nabla\phi\|_p \right] \frac{\left[ \hat{C}(t) \right]^{\frac{1}{p}}}{(t-s)^{\frac{n}{2p}}} \\
		&= C_{29} \frac{\hat{C}(t)^{\frac{1}{p}}}{(t-s)^{\frac{n}{2p}}}.
	\end{aligned}
\end{equation}

Consequently,
\begin{equation}
	m(t) \leq 1 + 2(\tau-1)C_{29}\hat{C}(T)^{\frac{1}{p}} \int_0^{t-\xi} \frac{m(s)}{(t-s)^{\frac{n}{2p}}} \, ds + 2(\tau-1)C_{29}\hat{C}(T)^{\frac{1}{p}} m(t)\xi^{1-\frac{n}{2p}}.
\end{equation}
Rearranging terms yields:
\begin{equation}
	\begin{aligned}
		\left[1 - 2(\tau-1)C_{29}\hat{C}(T)^{\frac{1}{p}}\xi^{1-\frac{n}{2p}}\right]m(t) &\leq 1 + 2(\tau-1)C_{29}\hat{C}(T)^{\frac{1}{p}} \int_0^{t-\xi} \frac{m(s)}{(t-s)^{\frac{n}{2p}}} \, ds \\
		&\leq 1 + 2(\tau-1)C_{29}\hat{C}(T)^{\frac{1}{p}}\xi^{-\frac{n}{2p}} \int_0^{t-\xi} m(s) \, ds.
	\end{aligned}
\end{equation}

Choosing $\xi = \left(4(\tau-1)C_{29}\hat{C}(T)^{\frac{1}{p}}\right)^{-\frac{2p}{2p-n}}$ gives:
\begin{equation}
	1 - 2(\tau-1)C_{29}\hat{C}(T)^{\frac{1}{p}}\xi^{1-\frac{n}{2p}} = \frac{1}{2},
\end{equation}
leading to:
\begin{equation}
	m(t) \leq 2 + \left(4(\tau-1)C_{29}\hat{C}(T)^{\frac{1}{p}}\right)^{\frac{2p}{2p-n}} \int_0^t m(s) \, ds.
\end{equation}
Gr\"onwall's inequality then implies:
\begin{equation}
	m(t) \leq 2 \exp\left( \left(4(\tau-1)C_{29}\hat{C}(T)^{\frac{1}{p}}\right)^{\frac{2p}{2p-n}} t \right).
\end{equation}

In particular, we obtain the global bound:
\begin{equation}
	w(x,t) \leq 2 \exp\left( \left(4(\tau-1)C_{29}\hat{C}(t)^{\frac{1}{p}}\right)^{\frac{2p}{2p-n}} t \right) \quad \forall t \in [0, \infty),
\end{equation}
which translates to:
\begin{equation}
	2^{-\frac{1}{\tau-1}} e^{-(\tau-1)^{\frac{n}{2p-n}} (4C_{29}\hat{C}(t)^{\frac{1}{p}})^{\frac{2p}{2p-n}} t} \leq J(x,t) \leq 1,
\end{equation}
\end{proof}
Let $Z = \frac{2-\delta}{2n}$. Applying the previous Lemma, we obtain
\begin{equation}
	(\Delta_\phi - \partial_t)Q + 2\langle\nabla w, \nabla Q\rangle \geq ZJ(\Delta_\phi w)^2 - \delta J|\nabla w|^4.
\end{equation}

Direct computation yields the expansion:
\begin{equation}
	\begin{aligned}
		(\Delta_\phi w)^2 &= \left( \frac{Q}{\alpha} + \frac{\alpha - J}{\alpha} |\nabla w|^2 \right)^2 \\
		&= \frac{Q^2}{\alpha^2} + 2\frac{Q(\alpha - J)}{\alpha^2}|\nabla w|^2 + \left( \frac{\alpha - J}{\alpha} \right)^2 |\nabla w|^4.
	\end{aligned}
\end{equation}

Substituting this into \eqref{Qestimate}, we derive
\begin{equation}
	\begin{aligned}
		(\Delta_\phi &- \partial_t)(tQ) + 2\langle\nabla w, \nabla(tQ)\rangle \\
		&\geq t\left[ ZJ\left( \frac{Q^2}{\alpha^2} + 2\frac{Q(\alpha - J)}{\alpha^2}|\nabla w|^2 + \left( \frac{\alpha - J}{\alpha} \right)^2 |\nabla w|^4 \right) - \delta J|\nabla w|^4 \right] - Q.
	\end{aligned}
\end{equation}

Assuming $Q \geq 0$ at its maximum point, we observe
\begin{equation}
	\frac{Q(\alpha - J)}{\alpha^2} |\nabla w|^2 \geq 0.
\end{equation}
Discarding non-negative terms, we simplify to
\begin{equation}
	\begin{aligned}
		(\Delta_\phi &- \partial_t)(tQ) + 2\langle\nabla w, \nabla(tQ)\rangle \\
		&\geq t\left[ \frac{ZJQ^2}{\alpha^2} + \left( ZJ\left( \frac{\alpha - J}{\alpha} \right)^2 - \delta J \right)|\nabla w|^4 \right] - Q.
	\end{aligned}
\end{equation}

To ensure non-negativity of the $|\nabla w|^4$ coefficient, we require
\begin{equation}
	Z\left( \frac{\alpha - J}{\alpha} \right)^2 - \delta \geq 0.
\end{equation}
Using the bound $J \leq 1$, we select $\delta = \frac{2}{2n+1}$ to obtain
\begin{equation}
	(\Delta_\phi - \partial_t)(tQ) + 2\langle\nabla w, \nabla(tQ)\rangle \geq \frac{t(2-\delta)J}{2n}\frac{Q^2}{\alpha^2} - Q.
\end{equation}

At the maximum point, this implies the critical bound
\begin{equation}
	Q \leq \frac{1}{t} \cdot \frac{2n}{(2-\delta)J},
\end{equation}
which finish the proof of the Theorem \ref{LYGEWHE}.

\begin{remark}
This Li-Yau type gradient estimates under lower Bakry-\'Emery curvature condition was first prove by Y.Li in \cite{Li2015} by classical method. Moreover, for the case $N > n$, an analogous estimate can be established under the weaker assumption $\|\operatorname{Ric}_\phi^N\|_{L^p(\mu)} \leq K$ (see \cite{OlivSeto}).
\end{remark}

In \cite{Ohta2021}, Ohta proposed an open question regarding the existence of Li-Yau-type estimates for weighted Riemannian manifolds $(M^n, g, \mu)$ satisfying a lower $N$-Ricci curvature bound in the $\varepsilon$-range with \textit{negative} dimension parameter $N < 0$. Our results partially address this question by imposing a weighted $L^p(\mu)$ bound on $|\nabla\phi|^2$. This leads us to formulate the following unresolved problem:

\begin{question}
	Under what geometric conditions can one establish \textit{natural bounds} for the weighted $L^p(\mu)$-norm of $|\nabla\phi|^2$ when $p > \frac{n}{2}$?
\end{question}
\section*{Acknowledgements}

Both authors would like to thank their supervisor Prof. Meng Zhu for inspiring discussions and invaluable suggestions.

\end{document}